\documentclass[12pt]{article}

\usepackage{amsmath,amssymb,amsthm,bm,enumitem,mathrsfs,mathtools,xcolor,mleftright,parskip,soul,tikz-cd,qtree}
\usepackage[colorlinks]{hyperref}
\usetikzlibrary{graphs}
\usepackage[utf8]{inputenc}
\usepackage[T1]{fontenc}
\usepackage{tikz}
\usetikzlibrary{trees}

\usepackage[margin = 1in]{geometry}

\setlist[itemize]{leftmargin = *}
\setlist[enumerate]{leftmargin = *}

\allowdisplaybreaks

\newtheorem{Cor}{Corollary}[section]
\newtheorem{Lem}[Cor]{Lemma}
\newtheorem{Thm}[Cor]{Theorem}
\newtheorem{Prop}[Cor]{Proposition}

\theoremstyle{remark}
\newtheorem{Eg}[Cor]{Example}
\newtheorem{Rmk}[Cor]{Remark}

\theoremstyle{definition}
\newtheorem{Def}[Cor]{Definition}

\newcommand{\C}{\mathbb{C}}
\newcommand{\D}{\displaystyle}
\newcommand{\G}{\mathcal{G}}
\newcommand{\M}[1]{\func{\mathcal{M}}{#1}}
\newcommand{\N}{\mathbb{N}}
\newcommand{\R}{\mathbb{R}}
\newcommand{\T}{\mathbb{T}}
\newcommand{\Z}{\mathbb{Z}}
\newcommand{\Br}[1]{\mleft( #1 \mright)}
\newcommand{\Cc}[1]{\func{\operatorname{C_{c}}}{#1}}

\newcommand{\df}{\vcentcolon =}
\newcommand{\Id}{\operatorname{Id}}
\newcommand{\ZZ}[2]{\func{Z_{\operatorname{cont}}^{1}}{#1,#2}}
\newcommand{\BB}[2]{\func{B_{\operatorname{cont}}^{1}}{#1,#2}}
\newcommand{\HH}[2]{\func{H_{\operatorname{cont}}^{1}}{#1,#2}}
\newcommand{\Abs}[1]{\mleft| #1 \mright|}

\newcommand{\DRG}[2]{\func{\G}{#1,#2}}
\newcommand{\End}[1]{\func{\operatorname{End}}{#1}}
\newcommand{\Int}[3]{\int_{#1} #2 ~ \d #3}

\newcommand{\KMS}[1]{\operatorname{KMS}_{#1}}

\newcommand{\Obj}[1]{\func{\operatorname{Obj}}{#1}}

\newcommand{\Seq}[2]{\Br{#1}_{#2}}
\newcommand{\Set}[2]{\SSet{#1 ~ \middle| ~ #2}}

\newcommand{\Card}[1]{\func{\mathsf{Card}}{#1}}
\newcommand{\Cont}[1]{\func{\operatorname{C}}{#1}}
\newcommand{\Diag}[1]{\func{\Delta}{#1}}
\newcommand{\func}[2]{#1 \Br{#2}}
\newcommand{\Func}[2]{\func{\Br{#1}}{#2}}
\newcommand{\FUNC}[2]{\func{\SqBr{#1}}{#2}}
\newcommand{\IInt}[4]{\Int{#1}{#2}{\func{#3}{#4}}}

\newcommand{\Pair}[2]{\Br{#1,#2}}

\newcommand{\SqBr}[1]{\mleft[ #1 \mright]}
\newcommand{\SSet}[1]{\mleft\{ #1 \mright\}}

\newcommand{\Trip}[3]{\Br{#1,#2,#3}}
\newcommand{\Borel}[1]{\func{\mathscr{B}}{#1}}
\newcommand{\Cstar}[1]{\func{C^{\ast}}{#1}}
\newcommand{\InvIm}[2]{#1^{- 1} \SqBr{#2}}

\newcommand{\Ruelle}[3]{\mathcal{L}_{#1,#2,#3}}
\newcommand{\CstarRed}[1]{\func{C^{\ast}_{\textnormal{r}}}{#1}}

\renewcommand{\d}{\mathrm{d}}
\renewcommand{\i}{\textnormal{i}}
\renewcommand{\ae}{\operatorname{a.e.}}
\renewcommand{\Im}[2]{#1 \SqBr{#2}}

\newcommand{\ZZintfor}[1]{\func{\mathcal{Z}}{#1}}

\begin{document}

%%%%%%%%%%%%%%%%%%%%%%%%%%%%%%%%%%%%%%%%%%%%%%%%%%%%%%%%%%%%%%%%%%%%%%%%%%%%%%%%%%%%%%%%%%%%%%%%%%%%%%%%%%%%%%%%%%%%%%%%%%%%%%%%%%%%%%%%%%%%%%%
% TITLE, ABSTRACT, AUTHOR INFORMATION %%%%%%%%%%%%%%%%%%%%%%%%%%%%%%%%%%%%%%%%%%%%%%%%%%%%%%%%%%%%%%%%%%%%%%%%%%%%%%%%%%%%%%%%%%%%%%%%%%%%%%%%%
%%%%%%%%%%%%%%%%%%%%%%%%%%%%%%%%%%%%%%%%%%%%%%%%%%%%%%%%%%%%%%%%%%%%%%%%%%%%%%%%%%%%%%%%%%%%%%%%%%%%%%%%%%%%%%%%%%%%%%%%%%%%%%%%%%%%%%%%%%%%%%%

\title{Cocycles on Certain Groupoids Associated to $\N^k$-Actions}
\author{Carla Farsi, Leonard Huang, Alex Kumjian \& Judith Packer}
\date{\today}

\maketitle

% \classification{46L05, 46L55, 46K10}

% \keywords{Cocyles, higher-rank Deaconu-Renault groupoids, Ruelle triples, Ruelle operaors, KMS-states.}

\begin{abstract}
We consider groupoids constructed from a finite number of commuting local homeomorphisms acting on a compact metric space, and study generalized Ruelle operators and $ C^{\ast} $-algebras associated to these groupoids. We provide a new characterization of $ 1 $-cocycles on these groupoids taking  values in a locally compact abelian group, given in terms of $ k $-tuples of continuous functions on the unit space satisfying certain canonical identities. Using this, we develop an extended Ruelle-Perron-Frobenius theory for dynamical systems of several commuting operators ($ k $-Ruelle triples and  commuting Ruelle operators). Results on KMS states on $ C^{\ast} $-algebras constructed from these groupoids are derived. When the groupoids being studied come from higher-rank graphs, our results recover existence-uniqueness results for KMS states associated to the graphs.
\end{abstract}

\tableofcontents

%%%%%%%%%%%%%%%%%%%%%%%%%%%%%%%%%%%%%%%%%%%%%%%%%%%%%%%%%%%%%%%%%%%%%%%%%%%%%%%%%%%%%%%%%%%%%%%%%%%%%%%%%%%%%%%%%%%%%%%%%%%%%%%%%%%%%%%%%%%%%%%
\section{Introduction} %%%%%%%%%%%%%%%%%%%%%%%%%%%%%%%%%%%%%%%%%%%%%%%%%%%%%%%%%%%%%%%%%%%%%%%%%%%%%%%%%%%%%%%%%%%%%%%%%%%%%%%%%%%%%%%%%%%%%%%%
%%%%%%%%%%%%%%%%%%%%%%%%%%%%%%%%%%%%%%%%%%%%%%%%%%%%%%%%%%%%%%%%%%%%%%%%%%%%%%%%%%%%%%%%%%%%%%%%%%%%%%%%%%%%%%%%%%%%%%%%%%%%%%%%%%%%%%%%%%%%%%%

Let $X$ be a compact Hausdorff space and $ \sigma: X \to X $ a local homeomorphism of $ X $ onto itself. The so-called Deaconu-Renault groupoid and its associated $C^*$-algebra  corresponding to the pair $ \Pair{X}{\sigma} $ were first studied by V. Deaconu in \cite{D} based on a construction of J. Renault in \cite{Re} in the setting of groupoids of Cuntz algebras. Deaconu adapted Renault's construction by replacing the shift map on the infinite sequence space with a local homeomorphism. Renault further generalized this construction to  local homeomorphisms defined on open subsets in \cite{ReClA}. The \'{e}tale  groupoid associated to a finite family of commuting local homeomorphisms of a compact metric space has gone by various names including Deaconu-Renault groupoids of higher rank, and semidirect product groupoid corresponding to the action of the semigroup $ \N^{k} $ \cite{ER}. In \cite{RWA}, these groupoids were generalized to the setting of partial semigroup actions. Our main purpose in this paper is the study of these groupoids, their associated $ C^{\ast} $-algebras and KMS states which arise naturally on a specific class of  related dynamical systems. This is a sufficiently broad  class to include higher-rank graph $C^*$-algebras associated to finite $k$-graphs. We develop cohomological methods to characterize $ 1 $-cocycles on these $ C^{\ast} $-algebras, which in turn give rise to one-parameter automorphism groups. This leads us to study the KMS states on these $ C^{\ast} $-algebras.
	
KMS states have their origin in equilibrium statistical mechanics and have long been a very fruitful tool in the study of operator algebras. In this paper, we study KMS states for groupoids associated to a finite family of commuting local homeomorphisms of a compact metric space by further developing a Ruelle-Perron-Frobenius (RPF) theory of dynamical systems of several commuting operators.  Although an RPF theory for free abelian semigroups has been introduced by M. Carvalho, F. Rodrigues, and P. Varandas in \cite{CFV1} and \cite{CFV2}, their main emphasis was on skew products, random walks, and topological entropy, whereas our emphasis here will be on the connection to the $ C^{\ast} $-algebras and the use of the Ruelle-Perron-Frobenius operator to prove the existence of measures with appropriate properties (hence states with related properties).

In the groupoid perspective, as first explained by Renault in \cite{Re}, time evolutions (dynamics) on the reduced $ C^{\ast} $-algebra of a groupoid $ \G $ are implemented by continuous real-valued $ 1 $-cocycles on $ \G $, and the task of understanding the KMS states for these dynamics on $ \CstarRed{\G} $ requires, at a minimum, identifying the measures on the unit space of $ \G $ that are quasi-invariant. There are now refinements of Renault's result, see for example work by S. Neshveyev \cite{N} and K. Thomsen \cite{Tho}. More recently, J. Christensen's paper \cite{C} combines quasi-invariant measures with a certain group of symmetries to describe KMS states on groupoid $ C^{\ast} $-algebras for locally compact second countable Hausdorff \'etale groupoids.

Our analysis of the KMS states on groupoids associated to a finite family of commuting local homeomorphisms of a compact metric space stems from a new characterization of their continuous real-valued $1$-cocycles, which in a nutshell are determined completely by a $ k $-tuple of continuous real-valued functions on the unit space of the groupoid satisfying canonical identities. In so doing, we give an isomorphism between the first monoid cohomology of $ \N^{k} $ with coefficients in the module $ \Cont{X,H} $ of continuous functions on $X$ with values in $H$, where $ H $ is a locally compact abelian group, and the first continuous cocycle groupoid cohomology taking values in $ H $.

We base our constructions on the established analysis of KMS states on Deaconu-Renault groupoids of (\cite{Exel, IK,KR,Re1}), together with an extended Ruelle-Perron-Frobenius theory for dynamical systems of several commuting operators, modeled on the one-dimensional theory of D. Ruelle \cite{Ru1} and P. Walters \cite{W1}.

In \cite{KR}, A. Kumjian and J. Renault  associated KMS states to Ruelle operators constructed on a groupoid arising from a single expansive map, and in \cite{IK}, M. Ionescu and A. Kumjian related the associated states to Hausdorff measures, which led to applications to KMS states on Cuntz algebras, $ C^{\ast} $-algebras arising from directed graphs, and $ C^{\ast} $-algebras associated to fractafolds. In addition, Ruelle operators were used in \cite{BEFT}. In this paper, we generalize some of these results to groupoids associated to  a  finite family of commuting local homeomorphisms of a compact metric space.  In particular we deduce that in order for the  adjoint of the Ruelle operator  associated to a finite family of commuting local homeomorphism to have an eigenmeasure, it is necessary and sufficient that the adjoint of the Ruelle operator corresponding to a non-trivial product of the local homeomorphisms have that same eigenmeasure, thus reducing matters to the one-dimensional case studied by P. Walters.

Ruelle operators are important tools in mathematical physics, particularly thermodynamics, and yield a formulation of a ``continuous'' extension of the seminal  Perron-Frobenius Theorem. Ruelle's classical  result, known as the ``Ruelle-Perron-Frobenius (RPF) Theorem'' gives a sufficient condition for a Ruelle triple to satisfy the unique positive eigenvalue condition  \cite{Ru1,Ru2}. In \cite{W}, building on earlier work of R. Bowen, P. Walters gave criteria for the RPF Theorem to hold for more general Ruelle triples $ \Trip{X}{\sigma}{\varphi} $ merely demanding that $ X $ be a metric space, $ \sigma $ be positively expansive and exact, and $ \varphi $ satisfy a smoothness condition. We extend the RPF Theorem to certain Ruelle triples of type $ \Trip{X}{\sigma}{\varphi} \df \Trip{X}{\sigma_{i}}{\varphi_{i}}_{i = 1}^{k} $, where the $ \sigma_{i} $ form a commuting family of local homeomorphisms which are positively expansive and exact, and the $ \varphi_{i} $ satisfy the Walters conditions.

To derive our generalization of the the RPF Theorem, we first need to construct continuous $ 1 $-cocycles on the groupoid $ \DRG{X}{\sigma} $ arising from a $ k $-tuple of commuting local homeomorphisms on $ X $, with values in $ \R $. In order to study this in the greatest possible generality, we first study the problem of calculating $ \func{H^{1}}{\N^{k},A} $ where $ A $ is an $ \N^{k} $-module. In Theorem \ref{prop:explicit description fo cocycles associated to cocycle condition}, we are able to give an explicit formula for all elements of $ \func{Z^{1}}{\N^{k},A} $ from $ k $-tuples in $ A^{k} $ satisfying what we call the ``module cocycle condition'' and describe which of these are coboundaries. In the case $ k = 1 $, this formula is similar to the formula given  in \cite{MDK}.

We  apply this theorem by starting with   a $ k $-tuple of commuting local homeomorphisms of $ X $, and using them to give  $ \Cont{X,H} $ an $ \N^{k} $-module structure, for any locally compact abelian group $ H $. We then provide an explicit isomorphism between $ \func{H^{1}}{\N^{k},\Cont{X,H}} $ and the first cohomology group $ \HH{\DRG{X}{\sigma}}{H} $ of the groupoid $ \DRG{X}{\sigma} $. Specializing to the case where $ H = \R $, we obtain explicit formulas for elements of  $ \ZZ{\DRG{X}{\sigma}}{\R} $ that we use in the generalized RPF theorem.

Recently there has been great interest (cf. \cite{ALRS,FGLP,McN}) in the KMS states associated to 1-parameter dynamical systems on $ \Cstar{\Lambda} $ where $ \Lambda $ is a higher-rank graph and the dynamics arises either from the canonical gauge action of $ \T^{k} $ on $ \Cstar{\Lambda} $ or from a generalized gauge action. In particular, for a finite strongly connected $ k $-graph, in \cite{ALRS,FGLP,McN}, one can endow $ \Cstar{\Lambda} $ with a (generalized) gauge dynamics and show the existence of unique KMS states. Here, we are able to recover some of the results in  \cite{ALRS,FGLP,McN} from a different perspective, using the Rulle-Perron-Frobenius theorem and the generalized gauge dynamics that we obtain from our description of $ \ZZ{\DRG{X}{\sigma}}{\R} $ given in Proposition \ref{cor:The Classification Theorem for Continuous Real-Valued 1-Cocycles on Deaconu-Renault Groupoids}.

In a follow-up paper, we plan to extend our results to topological $ k $-graphs.
	
We now outline the structure of the paper. Section \ref{sec:ruelle-triples-and-ops} introduces classical Ruelle triples, triples that satisfy the unique positive eigenvalue condition (see  Definition \ref{RPF Ruelle Triple}), and Ruelle operators. These are basic objects that we will generalize to higher-dimensions in Section \ref{sec:Ruelle-Dyn-systems}. We also  review several essential results of P. Walters, R. Bowen and D. Ruelle in this section. In Section \ref{sec:class-cocycles}, we review the construction of the groupoid $\DRG{X}{\sigma}$ associated to a finite family $\sigma$ of commuting local homeomorphisms of a compact metric space $X$, and then briefly review level one semigroup cohomology, continuous groupoid cohomology, and relate the two. We also give an algebraic way of constructing all continuous $ 1 $-cocycles  
both in the semigroup and groupoid case. Our main interest are continuous real-valued $ 1 $-cocycles on $ \DRG{X}{\sigma} $. In Section \ref{sec:Ruelle-Dyn-systems}, we introduce $ k $-Ruelle dynamical systems, the related families  of commuting Ruelle operators and their duals,  and their eigenmeasures.  In Section \ref{sec:Rad-Nyk-Meas-KMS}, we use the results of the previous sections to consider the Radon-Nikodym Problem for these groupoids, which provides a link between quasi-invariant measures for the groupoids $\DRG{X}{\sigma}$ and KMS states for a generalized gauge dynamics. In particular, we prove that if the generalized Ruelle operator associated to a $k$-Ruelle system has an eigenmeasure with eigenvalue $ 1 $, then there exists a KMS state for the generalized gauge dynamics coming from certain groupoid $ 1 $-cocycles related to the groupoid $ C^{\ast} $-algebra. Finally, in Section \ref{sec:gene-gauge-dyn-higher-rank-graphs}, we apply the results obtained thus far to answer some  existence and uniqueness questions concerning KMS states for a generalized gauge dynamics associated to higher-rank graphs.

%%%%%%%%%%%%%%%%%%%%%%%%%%%%%%%%%%%%%%%%%%%%%%%%%%%%%%%%%%%%%%%%%%%%%%%%%%%%%%%%%%%%%%%%%%%%%%%%%%%%%%%%%%%%%%%%%%%%%%%%%%%%%%%%%%%%%%%%%%%%%%%
\subsection{Notation and conventions}
%%%%%%%%%%%%%%%%%%%%%%%%%%%%%%%%%%%%%%%%%%%%%%%%%%%%%%%%%%%%%%%%%%%%%%%%%%%%%%%%%%%%%%%%%%%%%%%%%%%%%%%%%%%%%%%%%%%%%%%%%%%%%%%%%%%%%%%%%%%%%%%

In the sequel we will be using the  following notational conventions.  We denote by $ \N $ the semigroup of natural integers $ \SSet{0,1,2,\ldots} $, and by $ \N_{>0} $ the set of positive elements of $\N$. For fixed $ k \in \N_{>0} $,  we denote by     $\N^k$ the semigroup of all ordered k-tuples of elements of $\N$, and by  $ \SqBr{k} $ the set $\SqBr{k}=\SSet{1,\ldots,k} $.  We define the length $\Abs{n}$ of  an element $n \in \N^k$  by:   $\Abs{n} := n_{1} + \cdots + n_{k} $. 

For every compact Hausdorff space $ X $, and  topological locally compact group $H$, we let $ C(X,H) $ be the group of continuous functions from $X$ to $H$. In many instances $H=\R $ in this paper. 
We will also  let $ \M{X} $ denote the Banach space of finite signed Borel measures on the Borel subsets of $ X $, which is isometrically isomorphic to the dual space $ C(X,\R)' $ of $ C(X,\R) $. \footnote{By \cite[pp. 91--92]{WToolKit}, any measure on a locally compact and second countable space that is finite on compact sets is Radon, hence regular.}

For every (locally) compact Hausdorff \'{e}tale groupoid $ \G $, there is a standard dense linear embedding of $ \Cc{\G} $ into $ \CstarRed{\G} $. The groupoids that we study are amenable, so unless there is a danger of confusion, we shall identify $ f \in \Cc{\G} $ with its image (also denoted by $ f $) in $ \CstarRed{\G} \cong \Cstar{\G} $.
	
In what follows, $ X $ will always denote a (non-empty) compact Hausdorff topological space.

%%%%%%%%%%%%%%%%%%%%%%%%%%%%%%%%%%%%%%%%%%%%%%%%%%%%%%%%%%%%%%%%%%%%%%%%%%%%%%%%%%%%%%%%%%%%%%%%%%%%%%%%%%%%%%%%%%%%%%%%%%%%%%%%%%%%%%%%%%%%%%%
\subsection*{Acknowledgments}
%%%%%%%%%%%%%%%%%%%%%%%%%%%%%%%%%%%%%%%%%%%%%%%%%%%%%%%%%%%%%%%%%%%%%%%%%%%%%%%%%%%%%%%%%%%%%%%%%%%%%%%%%%%%%%%%%%%%%%%%%%%%%%%%%%%%%%%%%%%%%%%

This work was also partially supported by Simons Foundation Collaboration grants \#523991 (C.F.), \#353626 (A.K.) and \#316981 (J.P.). C.F. also thanks the sabbatical program at the University of Colorado Boulder for support.

The authors also thank the referee for helpful suggestions.

%%%%%%%%%%%%%%%%%%%%%%%%%%%%%%%%%%%%%%%%%%%%%%%%%%%%%%%%%%%%%%%%%%%%%%%%%%%%%%%%%%%%%%%%%%%%%%%%%%%%%%%%%%%%%%%%%%%%%%%%%%%%%%%%%%%%%%%%%%%%%%%
\section{Ruelle triples and Ruelle operators} \label{sec:ruelle-triples-and-ops} %%%%%%%%%%%%%%%%%%%%%%%%%%%%%%%%%%%%%%%%%%%%%%%%%%%%%%%%%%%%%%
%%%%%%%%%%%%%%%%%%%%%%%%%%%%%%%%%%%%%%%%%%%%%%%%%%%%%%%%%%%%%%%%%%%%%%%%%%%%%%%%%%%%%%%%%%%%%%%%%%%%%%%%%%%%%%%%%%%%%%%%%%%%%%%%%%%%%%%%%%%%%%%
	
We begin by defining  Ruelle triples and Ruelle operators (sometimes called {\it transfer operators}), which were introduced in \cite{Ru1} in the case of totally disconnected spaces, and generalized to arbitrary compact metric spaces by P. Walters in \cite{W1} and \cite{W}. These will be the basic objects of concern in this paper. Ruelle operators are important tools in mathematical physics, particularly thermodynamics, and yield a formulation of a ``continuous'' extension of the classical Perron-Frobenius Theorem.  In \cite{KR}, A. Kumjian and J. Renault  associated KMS states to Ruelle operators constructed on a groupoid arising from a single expansive map on a compact metric space, and in  \cite{IK}, M. Ionescu and A. Kumjian related the associated states to a Hausdorff measure on $ X $, which led to applications to KMS states on Cuntz $ C^{\ast} $-algebras, $ C^{\ast} $-algebras arising from directed graphs, and $ C^{\ast} $-algebras associated to fractafolds. \\

% ---------------------------------------------------------------------------------------------------------------------------------------------

\begin{Def}[Ruelle Triples and Operators] \label{Ruelle Triple}
\hfill
\begin{enumerate}
\item[(1)]
A \emph{Ruelle triple} is an ordered triple $ \Trip{X}{T}{\varphi} ,$ where
\begin{enumerate}
\item
$ X $ is a compact metric space.

\item
$ T : X \to X $ is a surjective local homeomorphism.

\item
$ \varphi: X \to \R $ is a continuous function, that is, $ \varphi \in \Cont{X,\R} $.
\end{enumerate}

\item[(2)]
The \emph{Ruelle operator} associated to a Ruelle triple $ \Trip{X}{T}{\varphi} $ is the bounded linear operator
$$
\Ruelle{X}{T}{\varphi}: \Cont{X,\R} \to \Cont{X,\R}
$$
defined by, for all $f \in \Cont{X,\R}, ~ \forall x \in X$: 
\begin{equation} \label{eq:def of ruelle operator}
\FUNC{\func{\Ruelle{X}{T}{\varphi}}{f}}{x} \df \sum_{y \in \InvIm{T}{\SSet{x}}} e^{\func{\varphi}{y}} \func{f}{y}.
\end{equation}
\end{enumerate}
\end{Def}

% ---------------------------------------------------------------------------------------------------------------------------------------------

Our goal is to extend some results from \cite{KR} and \cite{IK} from a single local homeomorphisms to commuting k-tuples of local homeomorphisms on $ X $ in part by employing cohomological methods. The following lemma follows from \cite[Proposition 2.2]{ER}. \\

% ---------------------------------------------------------------------------------------------------------------------------------------------

\begin{Lem}[Composition of Ruelle Operators] \label{A Formula for the Composition of Two Ruelle Operators}
Let $ \Trip{X}{S}{\varphi} $ and $ \Trip{X}{T}{\psi} $ be Ruelle triples. Then $ \Trip{X}{S \circ T}{\varphi \circ T + \psi} $  is a Ruelle triple, and
$$
\Ruelle{X}{S}{\varphi} \circ \Ruelle{X}{T}{\psi} = \Ruelle{X}{S \circ T}{\varphi \circ h + \psi}.
$$
\end{Lem}

% ----------------------------------------------------------------------------------------------------------------------------------------------
	
We will now define an important subclass of Ruelle triples, those for which the positive eigenvalue problem for the dual Ruelle operator has a unique solution. These Ruelle triples enjoy important fixed-point  properties, and admit generalizations to dynamical systems that will be described in Section \ref{sec:Ruelle-Dyn-systems}. \\
	
% ---------------------------------------------------------------------------------------------------------------------------------------------
	
\begin{Def}  \label{RPF Ruelle Triple}
A Ruelle triple $ \Trip{X}{T}{\varphi} $ is said to satisfy the  \emph{unique positive eigenvalue condition} if there exists a unique ordered pair $ \Pair{\lambda}{\mu} $ such that
\begin{enumerate}
\item[(1)]
$ \lambda$ is a positive real number.
			
\item[(2)]
$ \mu $ is a Borel probability measure on $ X $.
			
\item[(3)]
If we denote by  $ \Br{\Ruelle{X}{T}{\varphi}}^{\ast}: \M{X}  \to \M{X} $  the dual of the Ruelle operator $ \Ruelle{X}{T}{\varphi} $, then
\[ \func{\Br{\Ruelle{X}{T}{\varphi}}^{\ast}}{\mu} = \lambda \mu .\]
\end{enumerate}
\end{Def}

% ---------------------------------------------------------------------------------------------------------------------------------------------

Ruelle's classical  result, known as the ``Ruelle-Perron-Frobenius (RPF) Theorem'', generalizes the seminal Perron-Frobenius Theorem for primitive matrices to subshifts of finite type, and gives a sufficient condition for a Ruelle triple that satisfies Definition  \ref{RPF Ruelle Triple} \cite{Ru1,Ru2}. The RPF theorem below is taken from \cite[Theorem 2.2]{Exel}.

To introduce the required notation to state the RPF theorem, fixing $ k \in \N_{> 0} $, let $ A = \Seq{A_{i,j}}{i,j \in \SqBr{k}}$ be an $ \Br{n \times n} $ zero-one matrix with no row or column of zeros, and let $ \Pair{\Sigma_{A}}{\sigma} $ be the associated (one-sided) subshift of finite type, where $ \Sigma_{A} $ is the compact topological subspace of the infinite product space $ \prod_{j \in \N} \SqBr{k} $ defined by
$$
\Sigma_{A} \df \Set{x = \Br{x_{0},x_{1},x_{2},\ldots} \in \prod_{j \in \N} \SqBr{k}}{A_{x_{i},x_{i + 1}} = 1, ~ \forall i \geq 0},
$$
and $ \sigma: \Sigma_{A} \to \Sigma_{A} $ is the ``left shift'' given by
$$
\func{\sigma}{x_{0},x_{1},x_{2},\ldots} = \Br{x_{1},x_{2},x_{3},\ldots}.
$$
Moreover, given a real number $ \beta \in \Pair{0}{1} $, we define a compatible metric $ d $ on $ \Sigma_{A} $ by setting, for $ x,y \in \Sigma_{A} $ and $ x \neq y $, $ \func{d}{x,y} = \beta^{\func{N}{x,y}} $, where $ \func{N}{x,y} $ is the least integer $ N \in \N $ such that $ x_{i} \not=  y_{i} $. Furthermore  $ \func{\min}{\varnothing} \df \infty $ by convention.\\

We can now state the Ruelle-Perron-Frobenius Theorem as presented by R. Exel, see \cite[Theorem 2.2]{Exel} and \cite[Proposition 2.3]{Exel}.
% ---------------------------------------------------------------------------------------------------------------------------------------------

\begin{Thm}[Ruelle-Perron-Frobenius Theorem]  \label{thm:RPF-Theorem}
With notation as above let $ \varphi $ be a continuous real-valued function defined on $ \Sigma_{A} $. Suppose that
\begin{enumerate}

\item[(1)]
There exists a positive integer $ m $ such that $ A^{m} > 0 $ (in the sense that all entries are positive) and

\item[(2)]
$ \varphi $ is H\"{o}lder-continuous.
\end{enumerate}
Then there exist a strictly positive function $ h \in \Cont{\Sigma_{A},\R} $, 
a Borel probability measure $ \mu $ on $ \Sigma_{A} $, and a positive real number $ \lambda $ such that
\begin{enumerate}
\item[(a)]
$ \func{\Br{\Ruelle{\Sigma_A}{\sigma}{\varphi}}}{h} = \lambda h $ and

\item[(b)]
$ \func{\Br{\Ruelle{\Sigma_A}{\sigma}{\varphi}}^{\ast}}{\mu} = \lambda \mu $.
\end{enumerate}
In particular    $ \Trip{\Sigma_A}{\sigma}{\varphi} $  also satisfies the  unique positive eigenvalue condition of Definition \ref{RPF Ruelle Triple}.
\end{Thm}

% ---------------------------------------------------------------------------------------------------------------------------------------------

In the sequel, we will also refer to the Ruelle-Perron-Frobenius Theorem as the RPF Theorem. In \cite{W,W2}, P. Walters gave criteria for the RPF Theorem to hold for more general Ruelle triples $ \Trip{X}{T}{\varphi} $, which was modified by Kumjian and Renault in \cite{KR}, requiring that $ X $ be a metric space, $ T $ be positively expansive and exact, and that $ \varphi $ obey some summability condition. We will now detail these results. \\

% ---------------------------------------------------------------------------------------------------------------------------------------------

\begin{Def} \label{The Walters Criteria and the Jiang-Ye Conditions}
Let $ \Trip{X}{T}{\varphi} $ be a Ruelle triple, and let $ d $ be the metric on $ X $.  Consider the three conditions listed below. 

\begin{enumerate}
\item[(1)]
$ T $ is \emph{positively expansive}, i.e., there is an $ \epsilon > 0 $ such that for all distinct $ x,y \in X $, there exists an $ n \in \N $  such that $ \func{d}{\func{T^{n}}{x},\func{T^{n}}{y}} \geq \epsilon $.

\item[(2)]
$ T $ is \emph{exact}, i.e., for every non-empty open subset $ U $ of $ X $, there exists an $ n \in \N $ such that $ \Im{T^{n}}{U} = X $.

\item[(3)]
There exists a compatible metric $ d' $ on $ X $  and positive numbers $ \delta > 0 $ and $ C > 0 $ with the property that for all $n \in \N_{> 0} $ and  for all $ x,y \in X $ we have that 
$ \func{d'}{\func{T^{i}}{x},\func{{T}^{i}}{y}} \leq \delta $ for all $ i \in \SSet{0,1,\ldots,n - 1} $ implies
$$
\Abs{\sum_{i = 0}^{n - 1} \func{\varphi}{\func{T^{i}}{x}} - \func{\varphi}{\func{T^{i}}{y}}} \leq C.
$$
\end{enumerate}
\end{Def}

% ---------------------------------------------------------------------------------------------------------------------------------------------
We say that $ (X, T,\varphi)$ satisfies the {\it Walters conditions} if it satisfies condition (1) and (2) above, and it satisfies the {\it Bowen's condition} (\cite{Bow}) if it satisfies Condition (3) above.
Moreover, $ T $ is positively expansive if and only if  there is an open neighborhood $ U $ of $ \Diag{X} $, the diagonal of $ X $, such that for all distinct $ x,y \in X $, there exists an $ n \in \N $ such that $ \Pair{\func{T^{n}}{x}}{\func{T^{n}}{y}} \notin U $.

In all of our examples the  function $\varphi$  is H\"older continuous,  and that together with condition (1)
implies Bowen's condition, as noted in    the following well-known proposition   whose proof is sketched in \cite[pg. 2071]{KR}. \\

% ---------------------------------------------------------------------------------------------------------------------------------------------

\begin{Prop}\label{pos expansive  Holder} Let $ \Trip{X}{T}{\varphi} $ be a Ruelle triple, with $T$ positively expansive. 
If $ \varphi $ is H\"{o}lder-continuous with respect to a compatible metric $ d $ on $ X $, then condition (3) of Definition \ref{The Walters Criteria and the Jiang-Ye Conditions} is satisfied for $ d $ by $ \varphi $ with respect to $ T $.
\end{Prop}

% ---------------------------------------------------------------------------------------------------------------------------------------------

The main results of \cite{W} and \cite{JY} yield the following theorem. \\

% -------------------------------------------------------------------------------------------------------------------------
	
\begin{Thm}  \label{The Walters Criteria and the Jiang-Ye Criterion Imply the RPF Property}
A Ruelle triple satisfies the unique positive eigenvalue condition of 
Definition \ref{RPF Ruelle Triple} if it satisfies the conditions in Definition  \ref{The Walters Criteria and the Jiang-Ye Conditions}.
\end{Thm}

% ---------------------------------------------------------------------------------------------------------------------------------------------
	
In \cite{JY}, Y. Jiang and Y.-L. Ye stated analogous conditions for weakly contractive iterated-function systems for which results similar to Theorem \ref{The Walters Criteria and the Jiang-Ye Criterion Imply the RPF Property} hold.

%%%%%%%%%%%%%%%%%%%%%%%%%%%%%%%%%%%%%%%%%%%%%%%%%%%%%%%%%%%%%%%%%%%%%%%%%%%%%%%%%%%%%%%%%%%%%%%%%%%%%%%%%%%%%%%%%%%%%%%%%%%%%%%%%%%%%%%%%%%%%%%
\section{Continuous $ 1 $-cocycles on semigroups and continuous $ 1 $-cocycles on $ \DRG{X}{\sigma} $} \label{sec:class-cocycles} %%%%%%%%%%%%%
%%%%%%%%%%%%%%%%%%%%%%%%%%%%%%%%%%%%%%%%%%%%%%%%%%%%%%%%%%%%%%%%%%%%%%%%%%%%%%%%%%%%%%%%%%%%%%%%%%%%%%%%%%%%%%%%%%%%%%%%%%%%%%%%%%%%%%%%%%%%%%%

%%%%%%%%%%%%%%%%%%%%%%%%%%%%%%%%%%%%%%%%%%%%%%%%%%%%%%%%%%%%%%%%%%%%%%%%%%%%%%%%%%%%%%%%%%%%%%%%%%%%%%%%%%%%%%%%%%%%%%%%%%%%%%%%%%%%%%%%%%%%%%%
\subsection{Semigroup cocycles}
%%%%%%%%%%%%%%%%%%%%%%%%%%%%%%%%%%%%%%%%%%%%%%%%%%%%%%%%%%%%%%%%%%%%%%%%%%%%%%%%%%%%%%%%%%%%%%%%%%%%%%%%%%%%%%%%%%%%%%%%%%%%%%%%%%%%%%%%%%%%%%%

We now discuss semigroup cocycles with values a semigroup module $ A $, with the aim of explicitly constructing all $ \N^{k} $ $ 1 $-cocycles with values in the $ \N^{k} $-module $ A $. \\

% ---------------------------------------------------------------------------------------------------------------------------------------------

\begin{Def}[Semigroup Cocycles]
Let $ S $ be a semigroup and $ A $ an $ S $-module, so that $ A $ is an abelian group and there exists a homomorphism $ \pi: S \to \End{A} $. When there is no danger of confusion, for $ s \in S $ and $ \alpha \in A $, we denote by $ s \alpha \in A $ the element $ \FUNC{\func{\pi}{s}}{\alpha} $ of $ A $. Define $ \func{Z^{1}}{S,A} $ to be the set of $ A $-valued $ 1 $-cocycles on $ S $, that is, $ \func{Z^{1}}{S,A} $ is the set of functions
$$
\func{Z^{1}}{S,A} \df \Set{\gamma: S \to A}{\func{\gamma}{st} = \func{\gamma}{s} + s \func{\gamma}{t}, ~ \forall s,t \in S}.
$$
A function $ \gamma : S \to A $ is said to be an $ A $-valued $ 1 $-coboundary on $ S $ if there is an $ \alpha \in A $ such that $ \func{\gamma}{s} = \alpha - s \alpha $ for all $ s \in S $, in which case we write $\gamma = \gamma_{\alpha} $. Let $ \func{B^{1}}{S,A} $ denote the collection of all $ A $-valued $ 1 $-coboundaries on $ S $.
\end{Def}

% -------------------------------------------------------------------------------------------------------------------------
	
Routine computations show that $ \func{Z^{1}}{S,A} $ forms a group under addition, that every $ 1 $-coboundary is a $ 1 $-cocycle and that $ \func{B^{1}}{S,A} $ is a subgroup of $ \func{Z^{1}}{S,A} $. We verify that every $ 1 $-coboundary is in fact a $ 1 $-cocycle. Let $ \alpha \in A $ and $ s,t \in S $; then
\begin{align*}
    \func{\gamma_{\alpha}}{s t}
& = \alpha - \Br{s t}  \alpha
  = \alpha - \FUNC{\func{\pi}{st}}{ \alpha}
  = \alpha - \FUNC{\func{\pi}{s}}{ \alpha} + \FUNC{\func{\pi}{s}}{ \alpha} - \FUNC{\func{\pi}{s}}{\FUNC{\func{\pi}{t}}{ \alpha}} \\
& = \alpha - s  \alpha + s  \alpha - \FUNC{\func{\pi}{s}}{\FUNC{\func{\pi}{t}}{ \alpha}}
  = \alpha - s  \alpha + s  \alpha - s \Br{t  \alpha} \\
& = \func{\gamma_{\alpha}}{s} + s \func{\gamma_{ \alpha}}{t}.
\end{align*}
Hence, $ \func{B^{1}}{S,A} \subseteq \func{Z^{1}}{S,A} $. Moreover, we define the first semigroup cohomology of $ S $ with coefficients in $ A $ by $ \func{H^{1}}{S,A} \df \func{Z^{1}}{S,A} / \func{B^{1}}{S,A} $.

For the special case $ S = \N^{k} $, $k\in \N_{>0}$,  the following definition provides an important example of an $ S $-module. \\

% ---------------------------------------------------------------------------------------------------------------------------------------------
	
\begin{Def} \label{ex: special case RD groupoid module}
Let $ \sigma = \Seq{\sigma_{i}}{i \in \SqBr{k}} $ be a $ k $-tuple of commuting surjective local homeomorphism on the locally compact Hausdorff space $ X $. Let $ H $ be a topological locally compact abelian group. Define an $ \N^{k} $-module structure on $ A = \Cont{X,H} $ by setting
$$
\FUNC{\func{\pi_{n}}{f}}{x} \df \func{f}{\func{\sigma^{n}}{x}},
$$
for all $ n \in \N^{k} $, $ f \in \Cont{X,H} $ and $ x \in X $.
\end{Def}

% -------------------------------------------------------------------------------------------------------------------------

The next condition will be crucial in constructing $ 1 $-cocycles on $ \N^{k} $. \\

% -------------------------------------------------------------------------------------------------------------------------

\begin{Def}[Module Cocycle Condition] \label{def: module generators}
Let $ A $ be an $ \N^{k} $-module, and let $ a=\Seq{a_{i}}{i \in \SqBr{k}} $ be a $ k $-tuple of elements of $ A $. We say that $ \Seq{a_{i}}{i \in \SqBr{k}} $ satisfies  the {\it module cocycle condition} if for all $i,j \in [k], ~ i \neq j$: 
\begin{equation} \label{eq:module cocycle condition}
a_{i} +  \mathbf{e}_{i} a_{j} = a_{j} + \mathbf{e}_{j} a_{i},
\end{equation}
where $ \SSet{\mathbf{e}_{i}}_{i \in \SqBr{k}} $ are the canonical generators of $ \N^{k} $. \\
\end{Def}

% ---------------------------------------------------------------------------------------------------------------------------------------------

\begin{Thm} \label{prop:explicit description fo cocycles associated to cocycle condition} Let $ A $ be an $ \N^{k} $-module.
\begin{enumerate}
\item[(1)] Suppose that the  $ k $-tuple $ a = \Seq{a_{i}}{i \in \SqBr{k}} \in A^{k} $ satisfies the module cocycle condition of Equation \eqref{eq:module cocycle condition}. Then
there is a unique cocycle $ c_{a} \in \func{Z^{1}}{\N^{k},A} $ satisfying, for every $\ell \in \SqBr{k}$: 
\begin{equation} \label{eq:requirement general formula for the cocycle}
\func{c_{a}}{\mathbf{e}_{\ell}} = a_{\ell}.
\end{equation} 
 The cocycle $ c_{a} $ is given by the following formula:
\begin{equation} \label{eq:general formula for the cocycle}
  \func{c_{a}}{n} : = \sum_{i = 0}^{n_{1} - 1} \mathbf{e}_{1}^{i} a_{1} + \mathbf{e}_{1}^{n_{1}} \sum_{i = 0}^{n_{2} - 1} \mathbf{e}_{2}^{i} a_{2}\  +\  \cdots\  +\
  \mathbf{e}_{1}^{n_{1}} \mathbf{e}_{2}^{n_{2}} \cdots \mathbf{e}_{k - 1}^{n_{k - 1}} \sum_{i = 0}^{n_{k} - 1} \mathbf{e}_{k}^{i} a_{k}.
\end{equation}

\item[(2)]
The correspondence between the $ k $-tuples $ a = \Seq{a_{i}}{i \in \SqBr{k}} \in A^{k} $ satisfying the module cocycle condition and the associated  cocycles $ c_{a} \in \func{Z^{1}}{\N^{k},A} $ is a bijection.

\item[(3)]
Such a $ 1 $-cocycle $ c_{a} \in \func{Z^{1}}{\N^{k},A} $ corresponds to a coboundary in $ \func{B^{1}}{\N^{k},A} $ if and only there exists $\alpha \in A$ such that  $ a_{i} = \alpha - \mathbf{e}_{i} \alpha $ for $ i \in \SqBr{k} $.
\end{enumerate}
\end{Thm}

\begin{proof} Proof of (1). 
We subdivide the proof in two parts. We will first prove that the formula for $ c_{a} $ in Equation  \eqref{eq:general formula for the cocycle} gives a $ 1 $-cocycle on $ \N^{k} $ satisfying the conditions of Equation  \eqref{eq:requirement general formula for the cocycle}. Then we will prove the uniqueness.
	
For the fist part of the proof, we will proceed by induction. For a fixed $N\in \N$  our induction statement is that for any  $t,  m,n \in \N^{k} $, with  $ \Abs{t} \leq N $ and $ \Abs{m + n} \leq N $,  we have:

\begin{align}  \label{eq:new form ind hyp one} 
 \func{c_{a}}{t} &= \sum_{i = 0}^{t_{1} - 1} \mathbf{e}_{1}^{i} a_{1} + \mathbf{e}_{1}^{t_{1}} \sum_{i = 0}^{t_{2} - 1} \mathbf{e}_{2}^{i} a_{2}\  +\  \cdots\  +\
 \mathbf{e}_{1}^{t_{1}} \mathbf{e}_{2}^{t_{2}} \cdots \mathbf{e}_{k - 1}^{t_{k - 1}} \sum_{i = 0}^{t_{k} - 1} \mathbf{e}_{k}^{i} a_{k},\hbox{ and}\\  \label{eq:new form ind hyp two} 
\func{c_{a}}{m + n} &= \func{c_{a}}{m} + m\  \func{c_{a}}{n} .\\ \nonumber
\end{align}

The base case $N=1$ is easily checked as it amounts to, for  all $ \ell \in \SqBr{k} $:

\[ 
\begin{split}
\func{c_{a}}{\mathbf{e}_{\ell}} &=   \mathbf{e}_{1}^{0} \ldots  \mathbf{e}_{\ell -1}^{0} 
\mathbf{e}_{\ell}^{0} a_{\ell} = a_{\ell} ,\hbox{ and}\\
\func{c_{a}}{\mathbf{e}_{\ell}  } &= \func{c_{a}}{\mathbf{e}_{\ell}} + c_{a}(\mathbf{e}_{\ell} ) \, \func{c_{a}}{     0},\quad   \func{c_{a}}{\mathbf{e}_{\ell}  } = \func{c_{a}}{0} +0\,  c_{a}(\mathbf{e}_{\ell} ) . 
\end{split}\]

For the inductive step, we now suppose 
that the cocycle formula in   Equation  \eqref{eq:new form ind hyp one} holds for all $t\in \N^k$, with $|t|\leq N$, and that Equation  \eqref{eq:new form ind hyp two}  hold for all $n,m\in  \N^k$ with  $ |n+m|\leq N$. 
We need to show that  Equation  \eqref{eq:new form ind hyp one} holds for all $c_a(t)$, with $|t|\leq N+1$, and that Equation  \eqref{eq:new form ind hyp two}  hold for all $n,m\in  \N^k$ with  $ |n+m|\leq N+1$. 
To do so, fix $m, n \in \N$ with $\Abs{ m +n } = N $, and  choose any   $ \ell \in \SqBr{k} $ so that 
$ (m +n +\mathbf{e}_{\ell}) \in \N^k$, which implies  $\Abs{ m +n +\mathbf{e}_{\ell}}  = N + 1 $.

Assume first $m =0.$ Then since $|n| = N$,  the induction hypothesis (particularly Equation  \eqref{eq:new form ind hyp one}) implies:
\begin{equation} \label{eq:explicit cocycle equation}
  \func{c_{a}}{\mathbf{e}_{\ell}}\  +\  \mathbf{e}_{\ell}\, \func{c_{a}}{n}
= \func{c_{a}}{\mathbf{e}_{\ell}}\  +\ 
  \mathbf{e}_{\ell}
  \Br{
     \sum_{i = 0}^{n_{1} - 1} \mathbf{e}_{1}^{i} a_{1}\  +\  \cdots 
      \  +\ 
     \mathbf{e}_{1}^{n_{1}} \mathbf{e}_{2}^{n_{2}} \cdots \mathbf{e}_{k - 1}^{n_{k - 1}} \sum_{i = 0}^{n_{k} - 1} \mathbf{e}_{k}^{i} a_{k}
     }.
\end{equation}
Now note that for $ j < \ell $, the module cocycle condition of Equation \eqref{eq:module cocycle condition} implies that:
\begin{equation} \label{eq:formula for flipping}
	  \mathbf{e}_{\ell} \sum_{i = 0}^{n_{j} - 1} \mathbf{e}_{j}^{i} a_{j} = \sum_{i = 0}^{n_{j} - 1} \mathbf{e}_{j}^{i} \Br{\mathbf{e}_{\ell} a_{j}}
  = \sum_{i = 0}^{n_{j} - 1} \mathbf{e}_{j}^{i} \Br{a_{j} + \mathbf{e}_{j}  a_{\ell} - a_{\ell}}  =	\sum_{i = 0}^{n_{j} - 1} \mathbf{e}_{j}^{i} a_{j} + \mathbf{e}_{j}^{n_{j}} a_{\ell} - a_{\ell}.
\end{equation}
Next, we will we use Equation \eqref{eq:formula for flipping} to replace the terms $ \leq \ell $ in Equation \eqref{eq:explicit cocycle equation} with equivalent expressions, to get
\begin{equation} \label{eq:explicit cocycle equation after flipping}
\begin{split}
      \func{c_{a}}{\mathbf{e}_{\ell}} &+ \mathbf{e}_{\ell} \func{c_{a}}{n}
=  ~ a_{\ell} + \Br{\sum_{i = 0}^{n_{1} - 1} \mathbf{e}_{1}^{i} a_{1} + \mathbf{e}_{1}^{n_{1}} a_{\ell} - a_{\ell}} +
	    \mathbf{e}_{1}^{n_{1}} \Br{\sum_{i = 0}^{n_{2} - 1} \mathbf{e}_{2}^{i} a_{2} + \mathbf{e}_{j}^{n_{j}} a_{\ell} - a_{\ell}} + \cdots \\
  &\cdots + \mathbf{e}_{1}^{n_{1} - 1} \mathbf{e}_{2}^{n_{2}} \cdots \mathbf{e}_{\ell - 1}^{n_{\ell - 1}} \Br{\sum_{i=0}^{n_\ell } {\bf e}_\ell^i(a_\ell) -  a_\ell} + \cdots 
   + \mathbf{e}_{\ell} \mathbf{e}_{1}^{n_{1} - 1} \mathbf{e}_{2}^{n_{2}} \cdots \mathbf{e}_{k - 1}^{n_{k - 1}} \sum_{i = 0}^{n_{k} - 1} \mathbf{e}_{k}^{i} a_{k}.
\end{split}
\end{equation}
By using the telescopic properties of Equation \eqref{eq:explicit cocycle equation after flipping} above, one easily sees that $ \func{c_{a}}{\mathbf{e}_{\ell}} + \mathbf{e}_{\ell} \func{c_{a}}{n} $ is equal to
$$
\sum_{i = 0}^{n_{1} - 1} \mathbf{e}_{1}^{i} a_{1} + \mathbf{e}_{1}^{n_{1}} \sum_{i = 0}^{n_{2} - 1} \mathbf{e}_{2}^{i} a_{2} + \cdots +
\mathbf{e}_{1}^{n_{1}} \mathbf{e}_{2}^{n_{2}} \cdots \mathbf{e}_{\ell - 1}^{n_{\ell - 1}} \sum_{i = 0}^{n_{\ell}} \mathbf{e}_{\ell}^{i} a_{\ell} + \cdots +
\mathbf{e}_{1}^{n_{1}} \mathbf{e}_{2}^{n_{2}} \cdots \mathbf{e}_{k - 1}^{n_{k - 1}} \sum_{i = 0}^{n_{k} - 1} \mathbf{e}_{k}^{i} a_{k},
$$
which equals $ \func{c_{a}}{\mathbf{e}_{\ell} + n}$. That is we have proven that, for $n \in \N,$
$\Abs{n} = N $, and for any $ \ell \in \SqBr{k} $:
\begin{equation}\label{eq:final formula for m=0}  \func{c_{a}}{\mathbf{e}_{\ell}}\  +\  \mathbf{e}_{\ell}\, \func{c_{a}}{n}= \func{c_{a}}{\mathbf{e}_{\ell} + n}.\end{equation}

Moreover, by using Equation  \eqref{eq:new form ind hyp one} to replace $c_a(n)$ with the RHS of that equation in the above expression, a straightforward calculation shows that  Equation  \eqref{eq:new form ind hyp one}  holds for $t=n+   \mathbf{e}_{\ell}.$

We now suppose that $ m,n \in \N^{k} $ with $ \Abs{m + n} = N + 1 $ and $ \Abs{m} > 0, $ with positive $\ell-$th coordinate $m_\ell$, for some $ \ell \in \SqBr{k} $. Define  $ m' : =\ m- \mathbf{e}_{\ell}  \in \N^{k} $, and note that $ \Abs{\mathbf{e}_{\ell} + m' + n} = \Abs{m + n} - 1 = \Br{N + 1} - 1 = N $. Therefore by Equation  \eqref{eq:final formula for m=0}  we get:
$$
  \func{c_{a}}{m + n}
= \func{c_{a}}{\mathbf{e}_{\ell} + m' + n}
= \func{c_{a}}{\mathbf{e}_{\ell}} + \mathbf{e}_{\ell}\, \func{c_{a}}{m' + n}.
$$
By using the induction hypothesis twice we now get
\[
\begin{split}
    \func{c_{a}}{m + n}
& = \func{c_{a}}{\mathbf{e}_{\ell}} + \mathbf{e}_{\ell} \Br{\func{f}{m'} + m' \func{c_{a}}{n}}
  = \func{c_{a}}{\mathbf{e}_{\ell}} + \mathbf{e}_{\ell} \func{c_{a}}{m'} + \Br{\mathbf{e}_{\ell} + m} \func{c_{a}}{n} \\
& = \func{c_{a}}{\mathbf{e}_{\ell} + m'} + \Br{\mathbf{e}_{\ell} + m} \func{c_{a}}{n}
  = \func{c_{a}}{m} + m\, \func{c_{a}}{n}.
\end{split}
\]
Moreover, by using Equation  \eqref{eq:new form ind hyp one} to replace in the above expression $c_a(m)$ and $c_a(n)$  with the RHS of that equation, a straightforward calculation shows that   Equation  \eqref{eq:new form ind hyp one}  holds for $t=m+n.$

This completes the induction step, and so we have proven  that the formula for $ c_{a} $ in Equation  \eqref{eq:general formula for the cocycle} gives a $ 1 $-cocycle on $ \N^{k} $ satisfying the conditions of Equation  \eqref{eq:requirement general formula for the cocycle}.

We now prove the uniqueness. Let $ c_{a} $ be as described above, and let $ f $ be any other $ 1 $-cocycle on $ \N^{k} $ such that $ \func{f}{\mathbf{e}_{i}} = a_{i} $ for $ i \in \SqBr{k} $.  Now we proceed by induction on the length $ \Abs{n} $ of $n$. The base case that $ \func{c_{a}}{n} = \func{f}{n} $ for all $ n \in \N^{k} $ with $ \Abs{n} \leq 1 $ follows from the definition of $ c_{a} $ and $ f $.

For the inductive step, assume that $ \func{c_{a}}{r} = \func{f}{r} $ for all $ r \in \N^{k} $ with $ \Abs{r} \leq N $, and suppose $ \Abs{n} = n_{1} + n_{2} + \cdots + n_{k} = N $, $ \Abs{m} = m_{1} + m_{2} + \cdots + m_{k} = N + 1 $,  which implies $ m = n + \mathbf{e}_{\ell} $ for some $ \ell $. By using the inductive hypothesis and the module cocycle condition of Equation \eqref{eq:module cocycle condition}, we then get
$$
  \func{f}{m}
= \func{f}{\mathbf{e}_{\ell} + n}
= \func{c_{a}}{\mathbf{e}_{\ell}} + \mathbf{e}_{\ell} \func{c_{a}}{n}
= \func{c_{a}}{m}.
$$

Proof of (2).  We have already proven in (1) that the correspondence  $a \to c_a$ is injective.   To prove that $a \to c_a$ is a surjection onto $ \func{Z^{1}}{\N^{k},A} $, let us take $ c \in \func{Z^{1}}{\N^{k},A} $ and set $ a_{i} \df \func{c}{\mathbf{e}_{i}} $. It is then easy to check that the $ k $-tuple $ \Seq{a_{i}}{i \in \SqBr{k}} $ satisfies the module cocycle condition.

Proof of (2).  Suppose that the $ k $-tuple $a= \Seq{a_{i}}{i \in \SqBr{k}} \in A^{k} $ gives rise to a coboundary $ c_a $. Then by definition there exists $\alpha \in A$ such that for all $ i \in \SqBr{k}$: 
$$
a_{i}= \func{c_{a}}{\mathbf{e}_{i}}
= \alpha - \mathbf{e}_{i} \alpha .
$$
The other direction is clear.
\end{proof}

\begin{Rmk}\label{rmk:application of Theorem} In the most important of our uses of the above result, Theorem  \ref{prop:explicit description fo cocycles associated to cocycle condition}  applies to the $\N^k-$module  $A=C(X, \R)$ as in Definition  \ref{ex: special case RD groupoid module}.
\end{Rmk}
%%%%%%%%%%%%%%%%%%%%%%%%%%%%%%%%%%%%%%%%%%%%%%%%%%%%%%%%%%%%%%%%%%%%%%%%%%%%%%%%%%%%%%%%%%%%%%%%%%%%%%%%%%%%%%%%%%%%%%%%%%%%%%%%%%%%%%%%%%%%%%%
\subsection{Continuous $ 1 $-cocycles on $ \DRG{X}{\sigma} $}
%%%%%%%%%%%%%%%%%%%%%%%%%%%%%%%%%%%%%%%%%%%%%%%%%%%%%%%%%%%%%%%%%%%%%%%%%%%%%%%%%%%%%%%%%%%%%%%%%%%%%%%%%%%%%%%%%%%%%%%%%%%%%%%%%%%%%%%%%%%%%%%

In this subsection, our objective is to give an algebraic way of constructing all continuous $ H $-valued $ 1 $-cocycles, where $ H $ is a locally compact abelian group, on groupoids associated to a finite family of commuting local homeomorphisms of a compact metric space. In later sections we will mainly be interested in the case $ H = \R $. We begin by recalling the definition of these groupoids. \\

% ---------------------------------------------------------------------------------------------------------------------------------------------

\begin{Def}[\cite{D,ER}] \label{Deaconu-Renault Groupoid} 
Let $ \sigma = \Seq{\sigma_{i}}{i \in \SqBr{k}} $ be a $ k $-tuple of commuting surjective local homeomorphism on the  compact Hausdorff space $ X $. We regard $ \sigma $ as an action of $\mathbb{N}^{k} $ on $ X $ by the formula $ \sigma^{n} = \sigma_{1}^{n_{1}} \ldots \sigma_{k}^{n_{k}} $, where $ n = \Seq{n_{i}}{i \in \SqBr{k}} \in N^{k} $. The transformation groupoid (also called the the semi-direct product groupoid of the action) $ \DRG{X}{\sigma} $ is defined by
$$
\DRG{X}{\sigma} \df \Set{\Trip{x}{p - q}{y} \in X \times \Z^{k} \times X}{p,q \in \N^{k} ~ \text{and} ~ \func{\sigma^{p}}{x} = \func{\sigma^{q}}{y}}.
$$
We identify $ X $ with the unit space of $ \DRG{X}{\sigma} $ via the map $ x \mapsto \Trip{x}{0}{x} $. The structure maps are given by $ \func{r}{x,n,y} = x $, $ \func{s}{x,n,y} = y $, $ \Trip{x}{m}{y}^{- 1} = \Trip{y}{- m}{x} $, and $ \Trip{x}{m}{y} \Trip{y}{n}{z} = \Trip{x}{m + n}{z} $. A basis for the topology on $ \DRG{X}{\sigma} $ is given by subsets of the form
$$
\func{\mathcal{Z}}{U,V,m,n} \df U \times \SSet{p - q} \times V,
$$
where $ U,V $ are open in $ X $ and $ \Im{\sigma^{p}}{U} = \Im{\sigma^{q}}{V} $. We will denote by $ \DRG{X}{\sigma}^{\Br{2}} $ the set of composable pairs of $ \DRG{X}{\sigma} $.

The number $ k $ is called the \emph{rank} of $ \DRG{X}{\sigma} $.
\end{Def}
It is well known that $ \DRG{X}{\sigma} $ is an \'{e}tale locally compact Hausdorff amenable groupoid (cf. \cite{D,RWA}). \\

% ---------------------------------------------------------------------------------------------------------------------------------------------

\begin{Def}[Continuous Groupoid 1-Cocycles] \label{Groupoid 1-Cocycle}
Let $ \G $ be a topological  groupoid and $ H $ be a topological locally compact abelian  group.
A continuous  $ H $-valued \emph{$ 1 $-cocycle} on $ \G $ is a continuous function $ c: \G \to H $ such that for any $ \Pair{\gamma}{\gamma'} $ in  $ \G^{\Br{2}} $ we have
$$
\func{c}{\gamma \gamma'} = \func{c}{\gamma} + \func{c}{\gamma'}.
$$
In other words, $ c $ is just a continuous groupoid homomorphism from $ \G $ to $ H $.
We will denote by $ \ZZ{\G}{H} $ the set of continuous $ H $-valued $ 1 $-cocycles on $ \G $.
\end{Def}

% ---------------------------------------------------------------------------------------------------------------------------------------------

It is well known that $ \ZZ{\G}{H} $ is a group under pointwise addition and that $ \BB{\G}{H} $, the collection of continuous functions $ c: \G \to H $ such that there is a continuous function $ f : \G^{\Br{0}} \to H $ such that for all $ \gamma \in \G $, $ \func{c}{\gamma} = \func{f}{\func{r}{\gamma}} - \func{f}{\func{s}{\gamma}} $, is a subgroup of $ \ZZ{\G}{H} $. We define the first continuous cocycle groupoid cohomology of $ \G $ by
$$
\HH{\G}{H} := \ZZ{\G}{H} / \BB{\G}{H}.
$$
Our goal is to give an algebraic characterization of the cocycles in $ \ZZ{\DRG{X}{\sigma}}{H} $ for $ \DRG{X}{\sigma} $ expressed in  terms of their coordinate defining  functions as given in  (2) below. To do so, we will introduce the following definition. A special case of the module cocycle condition of Equation \eqref{eq:module cocycle condition} is \\

% ---------------------------------------------------------------------------------------------------------------------------------------------

\begin{Def}[The Cocycle Condition] \label{The Cocycle Condition}
Let $ H $ be a topological locally compact abelian group. Fix $ k \in \N $ and let $ \Trip{X}{\sigma}{\varphi} $ be an ordered triple with:
\begin{enumerate}
\item[(1)]
$ X $ a compact  metric space.

\item[(2)]
$ \sigma = \Seq{\sigma_{i}}{i \in \SqBr{k}} $ a  $ k $-tuple of commuting  surjective local homeomorphisms of $ X $.

\item[(3)]
$ \varphi = \Seq{\varphi_{i}}{i \in \SqBr{k}} $ a $ k $-tuple of elements from $ \Cont{X,H} $.
\end{enumerate}
Then $ \Trip{X}{\sigma}{\varphi} $ is said to satisfy the \emph{cocycle condition of order $ k $} if, for all $i,j \in \SqBr{k}$: 
\begin{equation}\label{eq:the cocycle condition}
\varphi_{i} + \varphi_{j} \circ \sigma_{i} = \varphi_{j} + \varphi_{i} \circ \sigma_{j}.
\end{equation}
Note that Equation \eqref{eq:the cocycle condition} is a special case  of the module cocycle condition of Definition \ref{def: module generators}.
When the order $ k $ is understood, we will omit it and just say that $ \Trip{X}{\sigma}{\varphi} $ satisfies the cocycle condition. Moreover, with a slight abuse of notation, when $X$ and $\sigma$ are understood, we will also say that $\varphi$ satisfies the cocycle condition.
\end{Def}

\begin{Eg}\label{ex:trivial cocycles}
	With notation as in Definition \ref{The Cocycle Condition}, if $ \varphi_{i} $ is constant for each $ i $, then $ \Trip{X}{\sigma}{\varphi} $ satisfies the cocycle condition.
\end{Eg}

The cocycle condition will be the characterizing feature for Ruelle triples in the case of a finite family of commuting endomorphisms, see Definition \ref{Ruelle Dynamical System}.

We now show that every groupoid cocycle $ c \in \ZZ{\DRG{X}{\sigma}}{H} $ arises from a $ k $-tuple of functions $ \Seq{\varphi_{i}}{i \in \SqBr{k}} $ satisfying the cocycle condition as above, and conversely. \\

% ---------------------------------------------------------------------------------------------------------------------------------------------

\begin{Prop}[Cocycle Characterization] \label{cor:The Classification Theorem for Continuous Real-Valued 1-Cocycles on Deaconu-Renault Groupoids}
Let $ \Trip{X}{\sigma}{\varphi} $ be a triple that satisfies conditions (1), (2) and (3) of Definition \ref{The Cocycle Condition}. Then statements (1) and (2) below are equivalent.
\begin{enumerate}
\item[(1)]
Algebraic characterization: $ \Trip{X}{\sigma}{\varphi} $ satisfies the cocycle condition.

\item[(2)]
There exists a unique $ c_{X,\sigma,\varphi} \in \ZZ{\DRG{X}{\sigma}}{H} $ such that $ \func{c}{x,\mathbf{e}_{i},\func{\sigma_{i}}{x}} = \func{\varphi_{i}}{x} $ for all $ i \in \SqBr{k} $ and $ x \in X $.

\end{enumerate}
Moreover,  every $ c \in \ZZ{\DRG{X}{\sigma}}{H} $ arises as a $ c_{X,\sigma,\varphi}$, for some $k-$tuple $\varphi$ satisfying the cocycle condition. In addition, $c \in \ZZ{\DRG{X}{\sigma}}{H} $ is a $ 1 $-coboundary if and only if there exists $ \psi \in \Cont{X,H}$ such that $ \func{c}{x,\mathbf{e}_{i},\func{\sigma_{i}}{x}} = \func{\psi}{x} - \func{\psi}{\func{\sigma_{i}}{x}} $, for all $i,j \in \SqBr{k}$.
\end{Prop}

\begin{proof}
$ \Br{1} \Longrightarrow \Br{2} $: Assume (1). Since $ \varphi=\Seq{\varphi}{i \in \SqBr{k}} $ satisfies the cocycle condition, by Proposition \ref{prop:explicit description fo cocycles associated to cocycle condition}, there is a unique $ 1 $-cocycle $ c_{\varphi} $ on $ \N^{k} $, taking values in the $ \N^{k} $-module $ \Cont{X,H} $, that satisfies $ \func{c}{\mathbf{e}_{i}} = \varphi_{i} $ for $ i \in \SqBr{k} $. Define $ c_{X,\sigma,\varphi} \in \ZZ{\DRG{X}{\sigma}}{H} $,  for all $m,n \in \N^{k}$, by: 
$$
\func{c_{X,\sigma,\varphi}}{x,m - n,y} \df \FUNC{\func{c_{\varphi}}{m}}{x} - \FUNC{\func{c_{\varphi}}{n}}{y}.
$$
We must show $ c_{X,\sigma,\varphi} $ is well defined, i.e., if there exist $ p,q \in \N^{k} $ with $ \Trip{x}{m - n}{y} = \Trip{x}{p - q}{y} \in \DRG{X}{\sigma} $, so that $\func{\sigma^{m}}{x} = \func{\sigma^{n}}{y} $ and $ \func{\sigma^{p}}{x} = \func{\sigma^{q}}{y} $, then
$$
\func{c_{X,\sigma,\varphi}}{x,m - n,y} = \func{c_{X,\sigma,\varphi}}{x,p - q,y}.
$$

For, note that since  $ m - n = p - q $, we have, for all $i \in \SqBr{k}$: 
$$
\func{m}{i} - \func{n}{i} = \func{p}{i} - \func{q}{i}, \quad
\func{m}{i} - \func{p}{i} = \func{n}{i} - \func{q}{i}.
$$
Define $ r \in \N^{k} $ by:
$$
r = \Br{\func{r}{1},\func{r}{2},\ldots,\func{r}{k}}, \quad \text{where} ~
\func{r}{i} =
\begin{cases}
\func{p}{i} - \func{m}{i} = \func{q}{i} - \func{n}{i}, & \text{if} ~ \func{p}{i} - \func{m}{i} > 0, \\
0,                                                     & \text{if} ~ \func{p}{i} - \func{m}{i} \leq 0.
\end{cases}
$$
It then follows:  $ m + r = m \vee p $ and $ n + r = q \vee p $, where  for $\alpha ,\beta \in \N^k$, we shall denote by $ \alpha \vee \beta $ the element of $ \N^{k} $ obtained by taking the max of the corresponding coordinates in $ \alpha $ and $ \beta $. Similarly to $r$ one can also define  $ t \in \N^{k} $ such that $ p + t = m \vee p $ and  $ q + t = n \vee q $. Hence  $ m + r = p + t $ and $ n + r = q + t $.

We now use the cocycle identity for $ c_{\varphi} $ to get, for all  $ \Trip{x}{m - n}{y} = \Trip{x}{p - q}{y} \in \DRG{X}{\sigma} $ :
\begin{align*}
    \FUNC{\func{c_{\varphi}}{m + r}}{x} - \FUNC{\func{c_{\varphi}}{n + r}}{y}
& = \FUNC{\func{c_{\varphi}}{m}}{x} + \FUNC{\func{c_{\varphi}}{r}}{\func{\sigma^{m}}{x}} - \Br{\FUNC{\func{c_{\varphi}}{n}}{y} + \FUNC{\func{c_{\varphi}}{r}}{\func{\sigma^{n}}{y}}} \\
& = \FUNC{\func{c_{\varphi}}{m}}{x} - \FUNC{\func{c_{\varphi}}{n}}{y} + \FUNC{\func{c_{\varphi}}{r}}{\func{\sigma^{m}}{x}} - \FUNC{\func{c_{\varphi}}{r}}{\func{\sigma^{n}}{y}},
\end{align*}
and because $ \func{\sigma^{m}}{x} = \func{\sigma^{n}}{y} $, the last two terms cancel each other out.

In the same way we show $ \FUNC{\func{c_{\varphi}}{p + t}}{x} - \FUNC{\func{c_{\varphi}}{q + t}}{y} = \FUNC{\func{c_{\varphi}}{p}}{x} - \FUNC{\func{c_{\varphi}}{q}}{y} $.

Therefore, $ \func{c_{X,\sigma,\varphi}}{x,m - n,y} = \func{c_{X,\sigma,\varphi}}{x,p - q,y} $, so $ c_{X,\sigma,\varphi} $ is well-defined. Moreover, using the fact that $ c_{\varphi} $ is a cocycle, it easily follows that $ c_{X,\sigma,\varphi} $ is a cocycle. The fact that $ c_{X,\sigma,\varphi} $ is continuous follows from the fact that $ c_{\varphi} $ takes on values in $ \Cont{X,H} $.

A straightforward calculation shows that $ c_{X,\sigma,\varphi} $ is unique. Indeed, any  $ c \in \ZZ{\DRG{X}{\sigma}}{H} $ is completely determined by  
$
\func{\varphi_{i}}{x}: = \func{c}{x,\mathbf{e}_{i},\func{\sigma^{\mathbf{e}_{i}}}{x}},\ i \in \SqBr{k},
$ and    $c=c_{X, \sigma,\varphi}$, where 
$\varphi =( \varphi_i) $
satisfies the cocycle condition. 

$ \Br{2} \Longrightarrow \Br{1} $: Assume (2). Then we have, for all $ i,j \in \SqBr{k} $ and $ x \in X $, that:
\begin{align*}
\Trip{x}{\mathbf{e}_{i}}{\func{\sigma_{i}}{x}} \Trip{\func{\sigma_{i}}{x}}{\mathbf{e}_{j}}{\func{\sigma_{j}}{\func{\sigma_{i}}{x}}}
& = \Trip{x}{\mathbf{e}_{i} + \mathbf{e}_{j}}{\func{\sigma_{j}}{\func{\sigma_{i}}{x}}} \\
& = \Trip{x}{\mathbf{e}_{j} + \mathbf{e}_{i}}{\func{\sigma_{i}}{\func{\sigma_{j}}{x}}} \\
& = \Trip{x}{\mathbf{e}_{j}}{\func{\sigma_{j}}{x}} \Trip{\func{\sigma_{j}}{x}}{\mathbf{e}_{i}}{\func{\sigma_{i}}{\func{\sigma_{j}}{x}}},
\end{align*}
and consequently, for all $ x \in X $,
\begin{align*}
    \func{\varphi_{i}}{x} + \func{\varphi_{j}}{\func{\sigma_{i}}{x}}
& = \func{c_{X,\sigma,\varphi}}{x,\mathbf{e}_{i},\func{\sigma_{i}}{x}} +
    \func{c_{X,\sigma,\varphi}}{\func{\sigma_{i}}{x},\mathbf{e}_{j},\func{\sigma_{j}}{\func{\sigma_{i}}{x}}} \\
& = \func{c_{X,\sigma,\varphi}}{\Trip{x}{\mathbf{e}_{i}}{\func{\sigma_{i}}{x}} \Trip{\func{\sigma_{i}}{x}}{\mathbf{e}_{j}}{\func{\sigma_{j}}{\func{\sigma_{i}}{x}}}} \\
& = \func{c_{X,\sigma,\varphi}}{x,\mathbf{e}_{j},\func{\sigma_{j}}{x}} +
    \func{c_{X,\sigma,\varphi}}{\func{\sigma_{j}}{x},\mathbf{e}_{i},\func{\sigma_{i}}{\func{\sigma_{j}}{x}}} \\
& = \func{\varphi_{j}}{x} + \func{\varphi_{i}}{\func{\sigma_{j}}{x}},
\end{align*}
which yields (1).

The statement about coboundaries is easily checked.
\end{proof}

% ---------------------------------------------------------------------------------------------------------------------------------------------

In the particular setting of a groupoid $ \DRG{X}{\sigma} $, with $ A = \Cont{X,H} $ endowed with an $\N^k$-module structure as in Definition \ref{ex: special case RD groupoid module}  and Remark \ref{rmk:application of Theorem}, Theorem \ref{prop:explicit description fo cocycles associated to cocycle condition} specializes to outline the relationship between $ \ZZ{\DRG{X}{\sigma}}{H} $ and $\N^k$-cocycles
with values in $A$. \\

% ---------------------------------------------------------------------------------------------------------------------------------------------

\begin{Cor} \label{prop: Alex semigroup cocycles correspondence}
Let $ \DRG{X}{\sigma} $ be a groupoid associated to a $k$-tuple $\sigma$ of commuting local homeomorphisms of a compact metric space $X$, and endow $ A = \Cont{X,H} $ with the structure of an $ \N^{k} $-module as in Definition \ref{ex: special case RD groupoid module}. Then there is an isomorphism
$$
\Phi: \func{Z^{1}}{\N^{k},A} \longrightarrow \ZZ{\DRG{X}{\sigma}}{H},\quad  \func{\Phi}{c_{\varphi}} : = c_{X, \sigma, \varphi} ,
$$
 where, by means of Theorem \ref{prop:explicit description fo cocycles associated to cocycle condition}, $ c_{\varphi}  \in \func{Z^{1}}{\N^{k},A} $ is determined by the $ k $-tuple $\varphi=  \Seq{\varphi_{i}}{i \in \SqBr{k}} \in A^{k} $, and $ c_{X, \sigma, \varphi} $ is the continuous groupoid cocycle associated to $ \Trip{X}{\sigma}{\varphi} $ as in Proposition \ref{cor:The Classification Theorem for Continuous Real-Valued 1-Cocycles on Deaconu-Renault Groupoids}.

Moreover, $\Phi$ restricts to an isomorphism  between coboundary groups, and in addition induces a first cohomology group  isomorphism
$$
\overline{\Phi}: \func{H^{1}}{\N^{k},A} \cong \HH{\DRG{X}{\sigma}}{H}.
$$
\end{Cor}

\begin{proof}
It is clear that $\Phi$ preserves the group operations between $ \func{Z^{1}}{\N^{k},A} $ and $ \ZZ{\DRG{X}{\sigma}}{H} $, and  Theorem \ref{prop:explicit description fo cocycles associated to cocycle condition} shows that $\Phi$ is a bijection. So $\Phi$ is an isomorphism.  It only remains to show that $\Phi$ induces an isomorphism between  $ \func{B^{1}}{\N^k,A} $ and  $\BB{\DRG{X}{\sigma}}{H} $. For, assume that $c_\varphi $ is a coboundary, that is,  there exists $ {f } \in A $ such that $c_\varphi =\gamma_f$, which implies  $ \varphi_{i} = f - \mathbf{e}_{i} f $ for all  $ i \in \SqBr{k} $. Define, for all $i \in \SqBr{k}$: 
$$
  \func{\varphi_{i}}{x}
= \func{f}{x} - \Func{\mathbf{e}_{i} f}{x}
= \func{f}{x} - \func{f}{\func{\sigma^{\mathbf{e}_i}}{x}}.
$$
 Then the cocycle  $c_f := \Phi(c_\varphi)$ is given by, on  $ \Trip{x}{l}{y} \in \DRG{X}{\sigma} $ for $ l = m - n $ with $ \func{\sigma^{m}}{x} = \func{\sigma^{m}}{y} $:
$$
\func{c_{f}}{x,l,y} = \func{f}{x} - \func{f}{y},
$$

which is obviously the coboundary corresponding to $ f \in \Cont{X,H} $. It is clear that $\Phi$ is one-to-one and onto from $ \func{B^{1}}{\N^{k},\Cont{X,H}} $ to $ \BB{\DRG{X}{\sigma}}{H} $. By the Fundamental Theorem for group homomorphisms, $\Phi$ also  induces an isomorphism
$$
\func{H^{1}}{\N^{k},\Cont{X,H}} \cong \HH{\DRG{X}{\sigma}}{H}. \qedhere
$$
\end{proof}

% ----------------------------------------------------------------------------------------------------------------------

%%%%%%%%%%%%%%%%%%%%%%%%%%%%%%%%%%%%%%%%%%%%%%%%%%%%%%%%%%%%%%%%%%%%%%%%%%%%%%%%%%%%%%%%%%%%%%%%%%%%%%%%%%%%%%%%%%%%%%%%%%%%%%%%%%%%%%%%%%%%%%%
\subsection{Cocycles and Ruelle operators}
%%%%%%%%%%%%%%%%%%%%%%%%%%%%%%%%%%%%%%%%%%%%%%%%%%%%%%%%%%%%%%%%%%%%%%%%%%%%%%%%%%%%%%%%%%%%%%%%%%%%%%%%%%%%%%%%%%%%%%%%%%%%%%%%%%%%%%%%%%%%%%%

We now use Proposition \ref{cor:The Classification Theorem for Continuous Real-Valued 1-Cocycles on Deaconu-Renault Groupoids} to deduce properties of  Ruelle operators corresponding to $ k $-tuples $\varphi=  \Seq{\varphi_{i}}{i \in \SqBr{k}} $ satisfying the cocycle condition. Recall that by Definition \ref{Ruelle Triple} the Ruelle operator associated to the Ruelle triple $ \Trip{\tilde{X}}{\tilde{\sigma}}{\tilde{\phi}} $ 
is denoted by  $ \Ruelle{\tilde{X}}{\tilde{\sigma}}{\tilde{\phi}} $. \\

% ---------------------------------------------------------------------------------------------------------------------------------------------
	
\begin{Thm}\label{The Workhorse Lemma} Let $X$ be a compact Hausdorff space,   ${\sigma} =( \sigma_i)_{{i \in \SqBr{k}}}$ be a 
k-tuple of commuting local homeomorphisms of $X$, and let 	$\varphi=  \Seq{\varphi_{i}}{i \in \SqBr{k}} $ be a k-tuple of elements of $A=\Cont{X,\R}$. 

\begin{enumerate}
	\item[(1)]  Assume that  $\varphi$ satisfies the cocycle condition, and let 
$ c_{\varphi}\in \func{Z^{1}}{\N^{k},A} $  be the   cocycle corresponding to $\varphi$ as in Theorem \ref{prop:explicit description fo cocycles associated to cocycle condition}. Then  the assignment
\[ \Psi: \N^{k} \to \End{\Cont{X,\R}},\quad  {\Psi}(n) := \Ruelle{X}{\sigma^{n}}{\func{c_{\varphi}}{n}}   \] is a semigroup homomorphism.

\item[(2)]  If the  k-tuple of Ruelle operators $\Big( \Ruelle{X}{\sigma^{\bf e_i}}{\varphi_i } \Big)_{i \in \SqBr{k}}$ commutes, then $\varphi$ satisfies the cocycle condition.
\end{enumerate}
\end{Thm}

\begin{proof}
For the proof of (1) see  Lemma \ref{A Formula for the Composition of Two Ruelle Operators}, which relies on \cite[Proposition 2.2]{ER} (this latter reference was pointed to us by the referee, whom we thank you for it).
	
To prove (2), let $ i,j \in \SqBr{k} $, $ x \in X $, and $ z = \Func{\sigma_{i} \circ \sigma_{j}}{x} $. As the set
$$
  S
= \InvIm{\Br{\sigma_{i} \circ \sigma_{j}}}{\SSet{z}}
= \InvIm{\Br{\sigma_{j} \circ \sigma_{i}}}{\SSet{z}}
$$
is finite, we can use Urysohn's Lemma to find an $ f \in \Cont{X,\R} $ such that $ \func{f}{x} = 1 $ and $ \func{f}{y} = 0 $ for all $ y \in S \setminus \SSet{x} $. Then
\begin{align*}
	  e^{\func{\varphi_{i}}{\func{\sigma_{j}}{x}} + \func{\varphi_{j}}{x}}
& = \FUNC{\Func{\Ruelle{X}{\sigma_{i}}{\varphi_{i}} \circ \Ruelle{X}{\sigma_{j}}{\varphi_{j}}}{f}}{z} \\
& = \FUNC{\Func{\Ruelle{X}{\sigma_{j}}{\varphi_{j}} \circ \Ruelle{X}{\sigma_{i}}{\varphi_{i}}}{f}}{z} \\
& = e^{\func{\varphi_{j}}{\func{\sigma_{i}}{x}} + \func{\varphi_{i}}{x}},
\end{align*}
which implies $ \func{\varphi_{i}}{\func{\sigma_{j}}{x}} + \func{\varphi_{j}}{x} = \func{\varphi_{j}}{\func{\sigma_{i}}{x}} + \func{\varphi_{i}}{x} $. As $ x \in X $ is arbitrary, we have proved Equation \eqref{eq:the cocycle condition}.
\end{proof}

% ---------------------------------------------------------------------------------------------------------------------------------------------

%%%%%%%%%%%%%%%%%%%%%%%%%%%%%%%%%%%%%%%%%%%%%%%%%%%%%%%%%%%%%%%%%%%%%%%%%%%%%%%%%%%%%%%%%%%%%%%%%%%%%%%%%%%%%%%%%%%%%%%%%%%%%%%%%%%%%%%%%%%%%%%
\section{Ruelle dynamical systems} \label{sec:Ruelle-Dyn-systems} %%%%%%%%%%%%%%%%%%%%%%%%%%%%%%%%%%%%%%%%%%%%%%%%%%%%%%%%%%%%%%%%%%%%%%%%%%%%%
%%%%%%%%%%%%%%%%%%%%%%%%%%%%%%%%%%%%%%%%%%%%%%%%%%%%%%%%%%%%%%%%%%%%%%%%%%%%%%%%%%%%%%%%%%%%%%%%%%%%%%%%%%%%%%%%%%%%%%%%%%%%%%%%%%%%%%%%%%%%%%%

We now introduce the main objects of our study: $ k $-Ruelle dynamical systems, which are higher-rank analogs of Ruelle triples. \\

% ---------------------------------------------------------------------------------------------------------------------------------------------
	
\begin{Def}[$ k $-Ruelle Dynamical Systems] \label{Ruelle Dynamical System}
A \emph{$ k $-Ruelle dynamical system} is an ordered triple $ \Trip{X}{\sigma}{\varphi} $ that satisfies conditions (1), (2), and (3) of Definition \ref{The Cocycle Condition} and the cocycle condition. For a $ k $-Ruelle dynamical system $ \Trip{X}{\sigma}{\varphi} $, we will denote by $ c_{X,\sigma,\varphi} $ the unique $ c \in \ZZ{\DRG{X}{\sigma}}{H} $ such that for all $ i \in \SqBr{k} $ and $ x \in X $, we have
$$
\func{c}{x,\mathbf{e}_{i},\func{\sigma_{i}}{x}} = \func{\varphi_{i}}{x}.
$$
Note that the existence of such a $ 1 $-cocycle is guaranteed by Proposition \ref{cor:The Classification Theorem for Continuous Real-Valued 1-Cocycles on Deaconu-Renault Groupoids}.
\end{Def}

% ---------------------------------------------------------------------------------------------------------------------------------------------

In Definition \ref{Ruelle Dynamical System}, we could have replaced the cocycle condition by any of its equivalent formulations in Proposition \ref{cor:The Classification Theorem for Continuous Real-Valued 1-Cocycles on Deaconu-Renault Groupoids}. However the cocycle condition usually is the easiest of the three equivalent conditions in Proposition \ref{cor:The Classification Theorem for Continuous Real-Valued 1-Cocycles on Deaconu-Renault Groupoids} to verify and work with.

In analogy with triples that satisfy the unique positive eigenvalue condition of  Definition \ref{RPF Ruelle Triple}, we define \\

% ---------------------------------------------------------------------------------------------------------------------------------------------

\begin{Def} \label{RPF Ruelle Dynamical System} 
A $ k $-Ruelle dynamical system $ \Trip{X}{\sigma}{\varphi} $ is said to \emph{admit a unique solution for the positive eigenvalue problem of the dual of the Ruelle operator} if there exists a unique ordered pair $ \Pair{\bm{\lambda}^{X,\sigma,\varphi}}{\mu^{X,\sigma,\varphi}} $  with the following properties:
\begin{enumerate}
\item[(1)]
 ${\bm{\lambda}^{X,\sigma,\varphi} }= \Seq{\lambda_{i}^{X,\sigma,\varphi}}{i \in \SqBr{k}} $ is a $ k $-tuple of positive real numbers.

\item[(2)]
$ \mu^{X,\sigma,\varphi} $ is a  Borel probability measure on $ X $.

\item[(3)] If $\big( {\Ruelle{X}{\sigma_{i}}{\varphi_{i}}}\big)^{\ast}$ denotes  the dual of the Ruelle operator, then 
for each $ i \in \SqBr{k} $:
\[ \func{\Br{\Ruelle{X}{\sigma_{i}}{\varphi_{i}}}^{\ast}}{\mu^{X,\sigma,\varphi}} = \lambda_{i}^{X,\sigma,\varphi} \mu^{X,\sigma,\varphi}.\]
\end{enumerate}
\end{Def}

% ---------------------------------------------------------------------------------------------------------------------------------------------

The next result is the $ k $-tuples version of Theorem \ref{thm:RPF-Theorem}. \\

% ---------------------------------------------------------------------------------------------------------------------------------------------

\begin{Thm} \label{The RPF Theorem for Ruelle Dynamical Systems}
A $ k $-Ruelle dynamical system $ \Trip{X}{\sigma}{\varphi} $ satisfies Definition \ref{RPF Ruelle Dynamical System} if there exists $ n \in \N^{k} \setminus \SSet{0} $ such that $ \Trip{X}{\sigma^{n}}{\func{c_{\varphi}}{n}} $ satisfies the  unique positive eigenvalue condition of  Definition \ref{RPF Ruelle Triple}.
\end{Thm}

\begin{proof}
Suppose that there is an $ n \in \N^{k} \setminus \SSet{0} $ such that $ \Trip{X}{\sigma^{n}}{\func{c_{\varphi}}{n}} $, with ${\varphi_i=c_{\varphi}}({\bf e_i})  $,  is a triple that satisfies the unique positive eigenvalue condition of Definition \ref{RPF Ruelle Triple}. Theorem \ref{The Workhorse Lemma} tells us that, for all $i \in [k]$ and $n \in \N$,
$\Ruelle{X}{\sigma^{\mathbf{e}_{i}}}{\varphi_{i}} $
and $ \Ruelle{X}{\sigma^{n}}{\func{c_{\varphi}}{n}} $ commute, which implies that  $ \Ruelle{X}{\sigma^{\mathbf{e}_{i}}}{\varphi_{i}}^{\ast} $ and $ \Ruelle{X}{\sigma^{n}}{\func{c_{\varphi}}{n}}^{\ast} $ also commute, and that:
\begin{align*}
    \func{\Br{\Ruelle{X}{\sigma^{n}}{\func{c_\varphi}{{n}}}}^{\ast}}{\func{\Br{\Ruelle{X}{\sigma_{i}}{\varphi_{i}}}^{\ast}}{\mu}}
& = \func{\Br{\Ruelle{X}{\sigma_{i}}{\varphi_{i}}}^{\ast}}{\func{\Br{\Ruelle{X}{\sigma^{n}}{\func{c_\varphi}{{n}}}}^{\ast}}{\mu}} \\
& = \func{\Br{\Ruelle{X}{\sigma_{i}}{\varphi_{i}}}^{\ast}}{\lambda  \mu}  = \lambda  \func{\Br{\Ruelle{X}{\sigma_{i}}{\varphi_{i}}}^{\ast}}{\mu}.
\end{align*}
Hence, $ \func{\Br{\Ruelle{X}{\sigma_{i}}{\varphi_{i}}}^{\ast}}{\mu} $ is an eigenmeasure of $ \Br{\Ruelle{X}{\sigma^{n}}{\func{c_\varphi}{{n}}}}^{\ast} $ with eigenvalue $ \lambda $. As we have for all $ f \in \Cont{X,\R_{\geq 0}} $ that
$$
	   \Int{X}{f}{\func{\Br{\Ruelle{X}{\sigma_{i}}{\varphi_{i}}}^{\ast}}{\mu}}
=    \Int{X}{\func{\Ruelle{X}{\sigma_{i}}{\varphi_{i}}}{f}}{\mu}
\geq 0,
$$
it follows that $ \func{\Br{\Ruelle{X}{\sigma_{i}}{\varphi_{i}}}^{\ast}}{\mu} $ is a non-negative measure on $ X $. Furthermore, by the surjectivity of $ \sigma_{i} $,
$$
  \Int{X}{1_{X}}{\func{\Br{\Ruelle{X}{\sigma_{i}}{\varphi_{i}}}^{\ast}}{\mu}}
= \Int{X}{\func{\Ruelle{X}{\sigma_{i}}{\varphi_{i}}}{1_{X}}}{\mu}
> 0,
$$
where we denote by $1_{X}$ the characteristic function of $X$.
 It follows that
$$
\frac{\func{\Br{\Ruelle{X}{\sigma_{i}}{\varphi_{i}}}^{\ast}}{\mu}}{\D \Int{X}{\func{\Ruelle{X}{\sigma_{i}}{\varphi_{i}}}{1_{X}}}{\mu}}  
$$
is a probability eigenmeasure of $ \Br{\Ruelle{X}{\sigma^{n}}{\func{c_\varphi}{{{n}}}}}^{\ast} $ with eigenvalue $ \lambda $. Since $ \Trip{X}{\sigma^{n}}{\func{c_{\varphi}}{n}} $ satisfies Definition \ref{RPF Ruelle Triple}, we must have
$$
\Pair{\lambda}{\frac{\func{\Br{\Ruelle{X}{\sigma_{i}}{\varphi_{i}}}^{\ast}}{\mu}}{\D \Int{X}{\func{\Ruelle{X}{\sigma_{i}}{\varphi_{i}}}{1_{X}}}{\mu}}} = \Pair{\lambda}{\mu},
$$
which yields
$$
\func{\Br{\Ruelle{X}{\sigma_{i}}{\varphi_{i}}}^{\ast}}{\mu} = \Br{\Int{X}{\func{\Ruelle{X}{\sigma_{i}}{\varphi_{i}}}{1_{X}}}{\mu}} \mu.
$$

Now if $ \alpha = \Seq{\alpha_{i}}{i \in \SqBr{k}} $ is a $ k $-tuple in $ \R_{> 0} $, and $ \nu $ a probability  Borel measure on $ X $ such that for all $ i \in \SqBr{k} $,
$$
\func{\Br{\Ruelle{X}{\sigma_{i}}{\varphi_{i}}}^{\ast}}{\nu} = \alpha_{i} \nu,
$$
then by Theorem \ref{The Workhorse Lemma}, we get, if we set $\alpha =(\alpha_1,\ldots, \alpha_k)$,
$$
\func{\Br{\Ruelle{X}{\sigma^{n}}{\func{c_{\varphi}}{n}}}^{\ast}}{\nu} = \alpha^{n} \nu.
$$

As $ \alpha^{n} > 0 $ and $ \Trip{X}{\sigma^{n}}{\func{c_\varphi}{n}} $ satisfies Definition  \ref{RPF Ruelle Triple}, we also obtain $ \Pair{\alpha^{n}}{\nu} = \Pair{\lambda}{\mu} $, so $ \nu = \mu.$  Hence, for all $ i \in \SqBr{k} $,
$$
\alpha_{i} \mu = \Br{\Int{X}{\func{\Ruelle{X}{\sigma_{i}}{\varphi_{i}}}{1_{X}}}{\mu}} \mu,
$$
which yields $ \alpha_{i} = \Int{X}{\func{\Ruelle{X}{\sigma_{i}}{\varphi_{i}}}{1_{X}}}{\mu} $ for all $ i \in \SqBr{k} $.
\end{proof}
	
% ---------------------------------------------------------------------------------------------------------------------------------------------

We will now give two examples of $ k $-Ruelle dynamical systems. \\

% ---------------------------------------------------------------------------------------------------------------------------------------------

\begin{Eg} \label{An RPF Ruelle Dynamical System on a Cantor Space}
For $ k \in \N_{> 1} $, let $ X \df \Z_{k}^{\N_{> 0}} $ be equipped with the product topology, where $ \Z_{k} \df \Z / k \Z $. It is well-known that the cylinder sets of $X$ form a basis for the topology on $X$; recall that every finite word   $\alpha \in (\Z_k)^m$ defines an associated  cylinder set  $\mathcal{Z}[\alpha]$ by:
\[\mathcal{Z}[\alpha] : = \left\{x\in X\ |\ (x_j)_{j=1,\ldots, m}= \alpha \right\}.\]
   Define a commuting $ k $-tuple $ \sigma = \Seq{\sigma_{i}}{i \in \SqBr{k}} $ of surjective local homeomorphisms on $ X $ by
$$
\forall i \in \SqBr{k}, ~ \forall x \in X: \qquad
\func{\sigma_{i}}{x} \df \Seq{x_{n + 1} + \Br{i - 1}}{n \in \N_{> 0}}.
$$
We want to verify  Definition  \ref{The Walters Criteria and the Jiang-Ye Conditions}. To do so we first check that $ \sigma_{i} $ is positively expansive and exact for each $ i \in \SqBr{k} $, and subsequently verify the $ \varphi $ is H\"{o}lder-continuous.

A straightforward calculation shows that, if we define
$ U :=\Set{  \Pair{x}{y} \in X \times X}{x_{1} = y_{1}} $, then:
$$
          \Diag{X}
\subseteq U
=         \bigcup_{i \in \SqBr{k}} \mathcal{Z}[i]  \times \mathcal{Z}[i] .
$$
To deduce that $ \sigma_{i} $ is positively expansive for each $ i \in \SqBr{k} $, simply observe that if $ x,y \in X $ are distinct, then $ \Pair{\func{\sigma_{i}^{m - 1}}{x}}{\func{\sigma_{i}^{m - 1}}{y}} \notin U $, where $ m \df \func{\min}{\Set{n \in \N_{> 0}}{x_{n} \neq y_{n}}} $.

Now we will show that $ \sigma_{i} $ is exact for each $ i \in \SqBr{k} $.   Let $U \subseteq X$ be a nonempty open set. Since the cylinder sets form a basis for the topology on $X$, there is a cylinder set $\mathcal{Z}[\alpha] \subseteq U,$
with $|\alpha|=m,$ for some $m \in \N_{>0}.$
Then, $ \Im{\sigma_{i}^{m}}{U} \supset \Im{\sigma_{i}^{m}}{\mathcal{Z}[\alpha] }= X $, which means that $ \sigma_{i} $ is exact for each $ i \in \SqBr{k} $.
		
Let $ d: X \times X\to \R_{\geq 0} $ denote the compatible metric on $ X $ defined by, for all $x,y \in X$: 
$$
\func{d}{x,y} \df 2^{- \func{\min}{\Set{n \in \N_{> 0}}{x_{n} \neq y_{n}}}},
$$
where $ \func{\min}{\varnothing} \df \infty $ by convention.

Now, let $ \Seq{a_{i}}{i \in \SqBr{k}} \in \R^{k} $, and define $ \varphi: X \to \R $ by,
for all $i \in \SqBr{k}, ~ \forall x \in X$: 
$$
\func{\varphi}{x} \df a_{i} \iff \sum_{n = 1}^{k} x_{n} = i - 1 ~ \Br{\text{mod} ~ k}.
$$
Clearly, $ \varphi $ is continuous. Indeed, we note that the value of $ \varphi $ at $ x\in X $ only depends on the first $ k $ components of $ x $, so that if $ x,y \in X$ and $ \func{d}{x,y} < 2^{- k} $, then $ \func{\varphi}{x} = \func{\varphi}{y} $ so that $ \Abs{\func{\varphi}{x} - \func{\varphi}{y}} = 0 $ in that case. One therefore computes, for all $ x,y \in X $,
$$
\Abs{\func{\varphi}{x} - \func{\varphi}{y}} \leq 2^{k} \Br{\max_{i,j \in \SqBr{k}} \Abs{a_{i} - a_{j}}} \func{d}{x,y},
$$
which implies that $ \varphi $ is H\"{o}lder-continuous with respect to $ d $.

Let $ \Seq{c_{i}}{i \in \SqBr{k}} \in \R^{k} $. Then for all $ i,j \in \SqBr{k} $ and for all $ x \in X $,
$$
  \Func{\varphi + c_{i} 1_{X}}{\func{\sigma_{j}}{x}} - \Func{\varphi + c_{j} 1_{X}}{\func{\sigma_{i}}{x}}
= \Br{\func{\varphi}{\func{\sigma_{i}}{x}} + c_{i}} - \Br{\func{\varphi}{\func{\sigma_{j}}{x}} + c_{j}}
= c_{i} - c_{j},
$$
the final equality holding because $ \func{\varphi}{\func{\sigma_{i}}{x}} = \func{\varphi}{\func{\sigma_{j}}{x}} $. Therefore, $ \Trip{X}{\Seq{\sigma_{i}}{i \in \SqBr{k}}}{\Seq{\varphi + c_{i} 1_{X}}{i \in \SqBr{k}}} $ is a Ruelle dynamical system, and since the Ruelle triple $ \Trip{X}{\sigma_{i}}{\varphi_{i}} $ satisfies conditions of Definition \ref{The Walters Criteria and the Jiang-Ye Conditions} for each $ i \in \SqBr{k} $, we conclude that $ \Trip{X}{\Seq{\sigma_{i}}{i \in \SqBr{k}}}{\Seq{\varphi + c_{i}  1_{X}}{i \in \SqBr{k}}} $ satisfies the unique positive eigenvalue condition of  Definition \ref{RPF Ruelle Triple}.
\end{Eg}

% ---------------------------------------------------------------------------------------------------------------------------------------------

The next example exhibits a dynamical system on a non-Cantor space, admitting a unique solution to the positive eigenvalue problem. \\

% ---------------------------------------------------------------------------------------------------------------------------------------------

\begin{Eg} \label{An RPF Ruelle Dynamical System on the Unit Circle}
Let $ n = \Seq{n_{i}}{i \in \SqBr{k}} $ be a $ k $-tuple, $ n_{i} \in \Z \setminus \SSet{0,\pm 1} $, and define a commuting $ k $-tuple $ \sigma = \Seq{\sigma_{i}}{i \in \SqBr{k}} $ of surjective local homeomorphisms on $ \T $ by, for all $i \in \SqBr{k},$ and for all $z \in \T$: $\func{\sigma_{i}}{z} \df z^{n_{i}}.$
The local homeomorphism $ \sigma_{i} $ is expansive for each $ i \in \SqBr{k} $ by \cite{Sh}, top of p. 176. Now let $ U $ be a non-empty open subset of $ \T $. Then there exist $ \alpha,\beta \in \R $ such that $ \alpha < \beta $ and
$$
\Set{e^{\i \theta} \in \T}{\alpha < \theta < \beta} \subseteq U.
$$
Let $ i \in \SqBr{k} $, and let $ m \in \N $ be such that $ 2 \pi \leq m \Abs{n_{i}} \Br{\beta - \alpha} $. Then
\begin{align*}
  \Im{\sigma_{i}^{m}}{U}
& \supseteq \Im{\sigma_{i}^{m}}{\Set{e^{\i \theta} \in \T}{\alpha < \theta < \beta}} 
 =         \Set{e^{\i m n_{i} \theta} \in \T}{\alpha < \theta < \beta} \\
& =         \Set{e^{\i \theta} \in \T}{\theta ~ \text{between} ~ m n_{i} \alpha ~ \text{and} ~ m n_{i} \beta} =         \T,
\end{align*}
so $ \sigma_{i} $ is exact.

Let $ d $ denote the  metric on $ \T $,  defined by, for all $ \alpha,\beta \in \R $:
\[d(e^{i\alpha},e^{i\beta})= \text{inf}\{|\alpha-(\beta+n\pi)| : n\in\mathbb Z\}.\]
The metric $ d $ generates the standard topology on $ \T $.

Define a $ k $-tuple $ \varphi = \Seq{\varphi_{i}}{i \in \SqBr{k}} $ in $ \Cont{X,\C} $ by, for all $i \in \SqBr{k}, ~ \forall z \in \T$:  
$\func{\varphi_{i}}{z} = z^{n_{i}} - z.$
A straightforward calculation shows that $ \varphi_{i} $ is H\"{o}lder-continuous with respect to $ d $ for each $ i \in \SqBr{k} $, and that the ($ \C $-valued) cocycle condition 
 $ \varphi_{i} + \varphi_{j} \circ \sigma_{i} = \varphi_{j} + \varphi_{i} \circ \sigma_{j}$  holds.

Now let $ f: \C \to \R $ be a continuous additive map (e.g., $ \func{f}{z} = {Re}(z) $). As $ f $ is then H\"{o}lder-continuous with respect to the Euclidean metrics on $ \C $ and $ \R $, $ f \circ \varphi_{i} $ is also H\"older-continuous with respect to $ d $, and by the additivity of $ f $, for every $ \forall i,j \in \SqBr{k} $,
$$
	\Br{f \circ \varphi_{i}} + \Br{f \circ \varphi_{j}} \circ \sigma_{i}
= f \circ \Br{\varphi_{i} + \varphi_{j} \circ \sigma_{i}}
= f \circ \Br{\varphi_{j} + \varphi_{i} \circ \sigma_{j}}
= \Br{f \circ \varphi_{j}} + \Br{f \circ \varphi_{i}} \circ \sigma_{j},
$$
so that $ \Trip{\T}{\sigma}{\Seq{f \circ \varphi_{i}}{i \in \SqBr{k}}} $ is a Ruelle dynamical system. As the Ruelle triple $ \Trip{\T}{\sigma_{i}}{f \circ \varphi_{i}} $ satisfies the conditions in Definition  \ref{The Walters Criteria and the Jiang-Ye Conditions} for each $ i \in \SqBr{k} $, we conclude that $ \Trip{\T}{\sigma}{\Seq{f \circ \varphi_{i}}{i \in \SqBr{k}}} $ has a unique eigenmeasure as in Definition \ref{RPF Ruelle Dynamical System}. \\
\end{Eg}

% ---------------------------------------------------------------------------------------------------------------------------------------------

\begin{Eg} \label{An RPF Ruelle Dynamical System w commuting matrices  on the circle}
Fixing $ d \in \N $ where $ d > 1 $, define two commuting local homeomorphisms $ \sigma_{1} $ and $ \sigma_{2} $ of $ \T^{2} $ by, for all $z_{1},z_{2} \in \T$: 
$$
\func{\sigma_{1}}{z_{1},z_{2}} \df \Pair{z_{1}^{d}}{z_{2}^{d}}, \quad
\func{\sigma_{2}}{z_{1},z_{2}} \df \Pair{z_{1} z_{2}^{- 1}}{z_{1} z_{2}}.
$$
Since all the eigenvalues of the associated matrices $ M_{\sigma_{1}} \df \begin{bmatrix} d & 0 \\ 0 & d \end{bmatrix} $ and $ M_{\sigma_{2}} \df \begin{bmatrix} 1 & - 1 \\ 1 & 1 \end{bmatrix} $ have modulus larger than $ 1 $, $ \sigma_{1} $ and $ \sigma_{2} $ are toral endomorphisms that are positively expansive and exact (see \cite{Mi}, proof of Theorem 1). In this case, we can choose the functions $ \varphi_{j} \df c_{j} 1_{\T^{2}} $ to be constant functions. Therefore the Ruelle triple $ \Trip{\T^{2}}{\sigma_{i}}{c_{i} 1_{\T^{2}}} $ satisfies the  conditions of Definition  \ref{The Walters Criteria and the Jiang-Ye Conditions}, for each $ i \in \SSet{1,2} $.
\end{Eg}

% ---------------------------------------------------------------------------------------------------------------------------------------------

The next example is a combination of Examples \ref{An RPF Ruelle Dynamical System on the Unit Circle} and \ref{An RPF Ruelle Dynamical System on a Cantor Space}. \\

% ---------------------------------------------------------------------------------------------------------------------------------------------

\begin{Eg} \label{An RPF Ruelle Dynamical System w examples 5.4 and 5.5}
Let $ X \df \Z_{k}^{\N_{> 0}} $, with $ k \in \N_{> 1} $. Fixing $ I \in \SqBr{k} $ and $ J \in \N_{> 1} $, define local homeomorphisms $ \sigma_{I},\sigma_{J}: X \times \T \to X \times \T $,   by the following formula: For all $ \Pair{x}{z} \in X \times \T $, set
\[
\begin{split}
\func{\sigma_{I}}{x,z} & \df \Pair{\Seq{x_{n + 1} + \Br{I - 1}}{n \in \N_{> 0}}}{z}, \\
\func{\sigma_{J}}{x,z} & \df \Pair{x}{z^{J}}.
\end{split}
\]
It is clear that $ \sigma_{I} $ and $ \sigma_{J} $ commute and
and one calculates that the composition $ \sigma_{I} \circ \sigma_{J} $ is positively expansive and exact. By proceeding as in Examples  \ref{An RPF Ruelle Dynamical System on the Unit Circle} and \ref{An RPF Ruelle Dynamical System on a Cantor Space}, and using Theorem \ref{The RPF Theorem for Ruelle Dynamical Systems} with $ k = 2 $ and $ \Pair{n_{1}}{n_{2}} = \Pair{1}{1} $, the resulting Ruelle triple satisfies the conditions in Definition  \ref{The Walters Criteria and the Jiang-Ye Conditions}. \\
\end{Eg}

% ---------------------------------------------------------------------------------------------------------------------------------------------

\begin{Eg}
We will now compute Ruelle eigenvalues and eigenmeasures for the $ k $-Ruelle dynamical system $ \Trip{X}{\sigma}{\varphi} $, with $ X = \prod_{j \in \N} \SSet{0,1} $, where $ \SSet{0,1} = \Z_2 $, and $ \sigma = \Pair{\sigma_{1}}{\sigma_{2}} $ defined by, for $ x = \Seq{x_{n}}{n \in \N} $
$$
\func{\sigma_{1}}{x} \df \Seq{x_{n + 1}}{n \in \N}, \quad \func{\sigma_{2}}{x} \df \Seq{x_{n} + 1}{n \in \N}.
$$
Moreover, for $ a,b,c \in \R $, define $ \varphi = \Pair{\varphi_{1}}{\varphi_{2}} $ by the following equation, where below addition is considered modulo $ 2 $:
\begin{equation}
\func{\varphi_{1}}{x} \df \left\{ \begin{array}{cl} a & \text{if} ~ x_{0} + x_{1} = 0, \\ b & \text{if} ~ x_{0} + x_{1} = 1, \end{array} \right. \quad
\func{\varphi_{2}}{x} \df c.
\end{equation}
Firstly, it is a simple exercise to determine that the eigenvalues $ \lambda_{1},\lambda_{2} $ of the associated Ruelle operator are given by $ \lambda_{1} = e^{a} +e^{b} $ and $ \lambda_{2} = e^{c} $, and that $ \func{\mu}{\func{Z}{0}} = \func{\mu}{\func{Z}{1}} = \frac{1}{2} $.

Moreover, the eigenmeasure $ \mu $ on all of the cylinder sets can be computed by using induction, thus proving that $ \mu $ is defined on the cylinder sets of $ X $ according to the following probability diagram and formula. We leave the details of this calculation to the reader.
$$
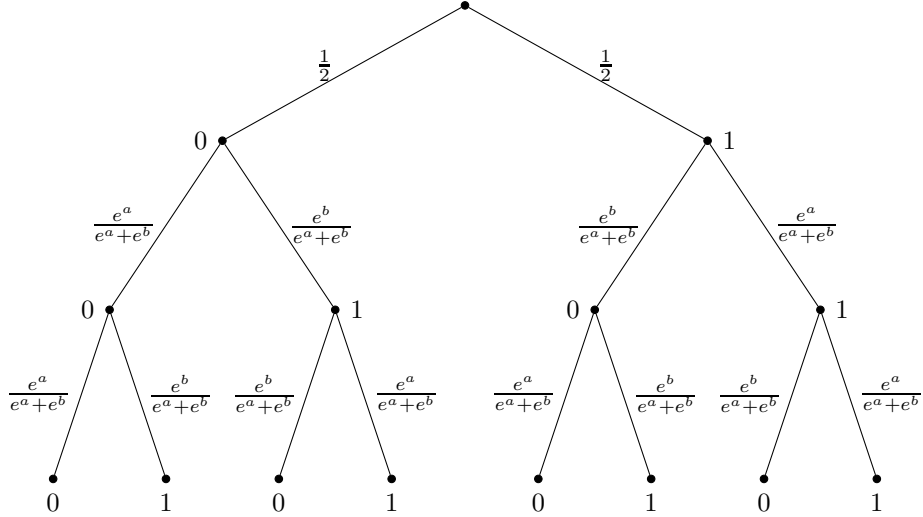
\begin{tikzpicture}
[
scale=1.5,
font=\footnotesize,
level 1/.style={level distance=12mm,sibling distance=43mm},
level 2/.style={level distance=15mm,sibling distance=20mm},
level 3/.style={level distance=15mm,sibling distance=10mm},
solid node/.style={circle,draw,inner sep=1,fill=black},
]

\node(0)[solid node]{}

child{node(1)[solid node,label=left:{$0$}]{}
	child{node[solid node,label=left:{$0$}]{}
		child{node[solid node,label=below:{$0$}]{} edge from parent node [left]{$\frac{e^a}{e^a +e^b}$}}
		child{node[solid node,label=below:{$1$}]{} edge from parent node [right]{$\frac{e^b}{e^a +e^b}$}}
		edge from parent node [left]{$\frac{e^a}{e^a +e^b}$}
	}
	child{node[solid node,label=right:{$1$}]{}
		child{node[solid node,label=below:{$0$}]{} edge from parent node [left]{$\frac{e^b}{e^a +e^b}$}}
		child{ node[solid node,label=below:{$1$}]{} edge from parent node [right]{$\frac{e^a}{e^a +e^b}$}}
		edge from parent node [right]{$\frac{e^b}{e^a +e^b}$}
	}
	edge from parent node [left, yshift=3]{$\frac{1}{2}$}
}
child{node(2)[solid node,label=right:{$1$}]{}
	child{node[solid node,label=left:{$0$}]{}
		child{node[solid node,label=below:{$0$}]{} edge from parent node [left]{$\frac{e^a}{e^a +e^b}$}}
		child{ node[solid node,label=below:{$1$}]{} edge from parent node [right]{$\frac{e^b}{e^a +e^b}$}}
		edge from parent node [left]{$\frac{e^b}{e^a +e^b}$}
	}
	child{node[solid node,label=right:{$1$}]{}
		child{node[solid node,label=below:{$0$}]{} edge from parent node [left]{$\frac{e^b}{e^a +e^b}$}}
		child{ node[solid node,label=below:{$1$}]{} edge from parent node [right]{$\frac{e^a}{e^a +e^b}$}}
		edge from parent node [right]{$\frac{e^a}{e^a +e^b}$}
	}
	edge from parent node [right, yshift=3]{$\frac{1}{2}$}
};
\end{tikzpicture}
$$
\begin{equation}\label{eq:formula-for-the-eigen-measure}
\func{\mu}{\func{Z}{x_{0} x_{1} \ldots x_{n}}} = \frac{1}{2} \prod_{j = 0}^{n - 1} \frac{e^{\func{\psi}{x_{j} + x_{j + 1}}}}{e^{a} +e^{b}},
\end{equation}
where $ \psi: \SSet{0,1} \to \SSet{a,b} $ is defined by $ \func{\psi}{0} \df a $ and $ \func{\psi}{1} \df b $.

In the above formula, note that when $ n = 0 $, the resulting product is empty, and so by convention equal to to $ 1 $. Therefore, $ \func{\mu}{\func{Z}{x_{0}}} = \frac{1}{2} $, for all $ x_{0} \in \SSet{0,1} $.
\end{Eg}

%%%%%%%%%%%%%%%%%%%%%%%%%%%%%%%%%%%%%%%%%%%%%%%%%%%%%%%%%%%%%%%%%%%%%%%%%%%%%%%%%%%%%%%%%%%%%%%%%%%%%%%%%%%%%%%%%%%%%%%%%%%%%%%%%%%%%%%%%%%%%%%
\section{The Radon-Nikodym Problem and KMS states} \label{sec:Rad-Nyk-Meas-KMS} %%%%%%%%%%%%%%%%%%%%%%%%%%%%%%%%%%%%%%%%%%%%%%%%%%%%%%%%%%%%%%%
%%%%%%%%%%%%%%%%%%%%%%%%%%%%%%%%%%%%%%%%%%%%%%%%%%%%%%%%%%%%%%%%%%%%%%%%%%%%%%%%%%%%%%%%%%%%%%%%%%%%%%%%%%%%%%%%%%%%%%%%%%%%%%%%%%%%%%%%%%%%%%%

This section addresses the Radon-Nikodym problem for groupoids associated to a finite family of commuting local homeomorphisms of a compact metric space, which provides a link between quasi-invariant measures for these groupoids and KMS states for a generalized gauge dynamics on the associated $ C^{\ast} $-algebra. As a result, there will be a heavier emphasis on measure theory and topology than the previous sections. \\

% ---------------------------------------------------------------------------------------------------------------------------------------------

\begin{Def}[Pull-Back and Quasi-Invariant Measures] \label{Quasi-Invariance}
Let $ \mu $ be a  Borel probability measure defined on the Borel sets of the compact metric space $ X $ with associated Borel $ \sigma $-algebra $ \Borel{X} $, and let $ \sigma = \Seq{\sigma_{i}}{i \in \SqBr{k}} $ be a commuting $ k $-tuple of surjective local homeomorphisms on $ X $. Define regular Borel measures $ s^{\ast} \mu $ and $ r^{\ast} \mu $ on $ \DRG{X}{\sigma} $ by, for all $ B \in \Borel{\DRG{X}{\sigma}} $:
\begin{align*}
\Func{s^{\ast} \mu}{B} & \df \IInt{X}{\Br{\sum_{\gamma \in \DRG{X}{\sigma}_{x}} \func{1_{B}}{\gamma}}}{\mu}{x}, \\
\Func{r^{\ast} \mu}{B} & \df \IInt{X}{\Br{\sum_{\gamma \in \DRG{X}{\sigma}^{x}} \func{1_{B}}{\gamma}}}{\mu}{x},
\end{align*}
where we denoted by $ \DRG{X}{\sigma}_{x} $ (resp. $ \DRG{X}{\sigma}^{x} $) the set of arrows in $ \DRG{X}{\sigma} $ with source $ x $ (resp. range $ x $). We then say that $ \mu $ is \emph{quasi-invariant} for $ \DRG{X}{\sigma}$ if $ s^{\ast} \mu $ and $ r^{\ast} \mu $ are equivalent to one another, in which case a \emph{Radon-Nikodym derivative} for $ \mu $ is any measurable function on $ \DRG{X}{\sigma} $ in the same equivalence class (with respect to the equivalence relation module sets of measure zero) as $ \frac{\d r^{\ast} \mu}{\d s^{\ast} \mu} $ \cite[Section 3]{Re}.
\end{Def}

% ---------------------------------------------------------------------------------------------------------------------------------------------

The following lemma will be used in the proof of Theorem \ref{A Characterization of Solutions of the Radon-Nikodym Problem}.  Its proof follows  from the definition of local homeomorphism and the fact that $ X $ is compact.\\

% ---------------------------------------------------------------------------------------------------------------------------------------------

\begin{Lem} \label{A Uniform Upper Bound for the Pre-Images of Points under Local Homeomorphisms on a Compact Metrizable Space}
Let $ T: X \to Y $ be a local homeomorphism of topological spaces from the compact metric space $ X $ to the metric space $ Y $. Then $ \sup_{y \in Y} \Card{\InvIm{T}{\SSet{y}}} < \infty $.
\end{Lem}

% ---------------------------------------------------------------------------------------------------------------------------------------------

The following theorem is a generalization of Proposition 4.2 of \cite{Re1} from Deaconu-Renault groupoids to groupoids $ \DRG{X}{\sigma} $, and characterizes the solutions of the Radon-Nykodym problem in this general setting. Its proof is quite technical, although in part it is possible to rely on the steps given in Renault's proof in \cite{Re1}. \\

% ---------------------------------------------------------------------------------------------------------------------------------------------

\begin{Thm} \label{A Characterization of Solutions of the Radon-Nikodym Problem}
Let $ \Trip{X}{\sigma}{\varphi} $ be a $ k $-Ruelle  dynamical system, which satisfies the conditions of Definition \ref{RPF Ruelle Dynamical System}, and $ \mu $ a Borel probability  measure on $ X $. Then the following statements are equivalent:
\begin{enumerate}
\item[(1)]
$ \mu $ is quasi-invariant for $\DRG{X}{\sigma} $, and the Radon-Nikodym derivative $ \frac{\d r^{\ast} \mu}{\d s^{\ast} \mu} $ is the continuous function $ e^{c_{X,\sigma,\varphi}} $ on $ \DRG{X}{\sigma} $.

\item[(2)]
$ \func{\Br{\Ruelle{X}{\sigma_{i}}{\varphi_{i}}}^{\ast}}{\mu} = \mu $ for each $ i \in \SqBr{k} $.
\end{enumerate}
\end{Thm}

\begin{proof}
Assume that (1) holds and  fix $ i \in \SqBr{k} $; then  $ \mu $ is quasi-invariant for the subgroupoid
$$
\DRG{X}{\sigma_{i}} \df \Set{\Trip{x}{\ell \mathbf{e}_{i}}{y}}{x,y \in X, ~ m,n \in \N, ~ \ell = m - n, ~ \func{\sigma_{i}^{m}}{x} = \func{\sigma_{i}^{n}}{y}}
$$
of $ \DRG{X}{\sigma} $ determined by the singly generated system $ \Pair{X}{\sigma_{i}} $. Thus, by Proposition 4.2 of \cite{Re1}, and using its notation, we have that
${}^{t} \func{\mathcal{L}_{\varphi_i}}{\mu} = \mu,$
which  in our notation means  $ \func{\Br{\Ruelle{X}{\sigma_{i}}{\varphi_{i}}}^{\ast}}{\mu} = \mu $. So $ \Br{1} \Longrightarrow \Br{2} $.

Conversely, assume (2) holds. We first briefly explain why $ \mu $ is quasi-invariant for $ \DRG{X}{\sigma} $. A straightforward calculation shows that for all $ i \in \SqBr{k} $, $ f \in \Cont{X,\R} $, and $ x \in X $, we have
\begin{gather*}
  \Int{X}{f}{\Br{\Br{\sigma_{i}}_{\ast} \mu}}
= \Int{X}{f  \func{\Ruelle{X}{\sigma_{i}}{\varphi_{i}}}{1_{X}}}{\mu}.
\end{gather*}
The Riesz representation theorem then implies  that for all $ A \in \Borel{X} $,
$$
\Func{\Br{\sigma_{i}}_{\ast} \mu}{A} = \Int{X}{1_{A}  \func{\Ruelle{X}{\sigma_{i}}{\varphi_{i}}}{1_{X}}}{\mu},
$$
which yields $ \Br{\sigma_{i}}_{\ast} \mu = \mu \circ \sigma_{i}^{- 1} \ll \mu $ for $ i \in \SqBr{k} $. On the other hand,  $ \mu \ll \mu \circ \sigma_{i}^{- 1}  $ for $ i \in \SqBr{k} $ by Proposition 4.2 of \cite{Re1}. By Lemma \ref{A Uniform Upper Bound for the Pre-Images of Points under Local Homeomorphisms on a Compact Metrizable Space}, we now  get that  for all $ B \in \Borel{\DRG{X}{\sigma}} $ and some $N\in \N_{>0}$,
$$
     \func{\mu}{\Im{s}{B}}
\leq \Func{s^{\ast} \mu}{B}
\leq N \func{\mu}{\Im{s}{B}},
$$
which implies that $ \func{\mu}{\Im{s}{B}} = 0 $ if and only if $ \Func{s^{\ast} \mu}{B} = 0 $; using a similar technique, we can also prove that $ \func{\mu}{\Im{r}{B}} = 0 $ if and only if $ \Func{r^{\ast} \mu}{B} = 0 $. Finally, by the monotonicity and countable additivity of $ \mu $, we have for all $ B \in \Borel{\DRG{X}{\sigma}} $, that:
$$
     \func{\mu}{\Im{s}{B}}
\leq \func{\mu}{\bigcup_{m,n \in \N^{k}} \Im{\Br{\sigma^{m}}^{- 1}}{\Im{\sigma^{n}}{\Im{r}{B}}}}
=    0,
$$
so that if $ \func{\mu}{\Im{r}{B}} = 0 $, then $ \func{\mu}{\Im{s}{B}} = 0 $, and the same method shows that if $ \func{\mu}{\Im{s}{B}} = 0 $, then $ \func{\mu}{\Im{r}{B}} = 0 $. All of these facts taken together imply that for a fixed $ B \in \Borel{\DRG{X}{\sigma}} $, we have $ \Func{r^{\ast} \mu}{B} = 0 $ if and only if $ \Func{s^{\ast} \mu}{B} = 0 $, so $ r^{\ast} \mu $ and $ s^{\ast} \mu $ are equivalent Borel measures on $ \DRG{X}{\sigma} $. Therefore, $ \mu $ is quasi-invariant for $ \DRG{X}{\sigma} $.

Now  by  (2), we have for all $ f \in \Cont{X,\R} $ and every $ i \in \SqBr{k} ,$ that:
\begin{align*}
		\Int{X}{f}{\mu}
& = \Int{X}{f}{\func{\Br{\Ruelle{X}{\sigma_{i}}{\varphi_{i}}}^{\ast}}{\mu}} = \Int{X}{\func{\Ruelle{X}{\sigma_{i}}{\varphi_{i}}}{f}}{\mu} \\
& = \IInt{X}{\Br{\sum_{y \in \InvIm{\sigma_{i}}{\SSet{x}}} e^{\func{\varphi_{i}}{y}} \func{f}{y}}}{\mu}{x} \\
& = \IInt{X}{\Br{\sum_{\gamma \in \DRG{X}{\sigma}_{x}} e^{\func{\varphi_{i}}{\func{r}{\gamma}}} \func{f}{\func{r}{\gamma}} \func{1_{S_{i}}}{\gamma}}}{\mu}{x} \\
& = \IInt{\DRG{X}{\sigma}}{e^{\func{\varphi_{i}}{\func{r}{\gamma}}} \func{f}{\func{r}{\gamma}} \func{1_{S_{i}}}{\gamma}}{\Br{s^{\ast} \mu}}{\gamma} \\
& = \IInt{\DRG{X}{\sigma}}{e^{\func{c_{X,\sigma,\varphi}}{\gamma}} \func{f}{\func{r}{\gamma}} \func{1_{S_{i}}}{\gamma}}{\Br{s^{\ast} \mu}}{\gamma} \\
& = \IInt{\DRG{X}{\sigma}}{\func{f}{\func{r}{\gamma}} \func{1_{S_{i}}}{\gamma}  e^{\func{c_{X,\sigma,\varphi}}{\gamma}}}{\Br{s^{\ast} \mu}}{\gamma}.
\end{align*}
From the above equations  it follows that
the  Radon-Nikodym derivative of $ r^{\ast} \mu $ with respect to $ s^{\ast} \mu $ must be equal to  $  e^{\func{c_{X,\sigma,\varphi}}{\gamma}} $ for almost all $ \gamma \in S_{i} $. By \cite{Re}, Proposition I.3.3, $ D $ is a measurable $ \R\setminus \{0\} $-valued $ 1 $-cocycle on $ \DRG{X}{\sigma},$ and $ D =_{\ae} e^{c_{X,\sigma,\varphi}} $, which is continuous by assumption. Therefore, $ e^{c_{X,\sigma,\varphi}} $ is a continuous Radon-Nikodym cocycle associated to the quasi-invariant measure $ \mu $, and we have established $ \Br{2} \Longrightarrow \Br{1} $.
\end{proof}

%---------------------------------------------------------------------------------------------------------------------------------------------

We now  illustrate a particular problem of existence of KMS states arising in the context of $ \DRG{X}{\sigma}$.  In the following example there are no KMS states associated to the dynamics, even though one of the associated local homeomorphisms acting on $X$ is expansive and exact. \\

%---------------------------------------------------------------------------------------------------------------------------------------------
	
\begin{Eg} \label{ex:tensor-product-Cuntz-algebra}
	Recall that the Cuntz algebra $ \mathcal{O}_{N} $, where $ N \geq 2 $,  is  the $ C^{\ast} $-algebra associated to the groupoid  arising from the  action of  the standard shift $ \sigma_{N} $ on $ X_{N} \df \prod_{j \in \N} \SqBr{N} $.
In \cite{OP}, D. Olesen and G. Pedersen prove that  for $N \geq 2$ there is exactly one KMS state for $ \mathcal{O}_{N} $ with respect to the canonical gauge action $ \alpha_{N} $ of $ \R $ on $ \mathcal{O}_{N} $ associated  to the cocycle determined by $ \varphi_{N} = 1 $. This KMS state arises at the inverse temperature value $ \beta = \func{\ln}{N} $. 

Now take $ X = X_{2} \times X_{3} $, and define $ \sigma = \Pair{\sigma_{2}}{\sigma_{3}} $, where $ \sigma_{j} $ is the standard shift on $ X_{j} $; also set $ \varphi = \Pair{\varphi_{2}}{\varphi_{3}} $, with $\varphi_{2} = 1$ and $\varphi_{3} = 1$. Note that the shift corresponding to $ \Pair{1}{1} \in \N^{2} $ is expansive and exact. Consider the automorphism group $ \alpha = \alpha_{2} \otimes \alpha_{3} $ defined on  the $ C^{\ast} $-algebra corresponding to $ \Trip{X}{\sigma}{\varphi} $, which is the tensor product $ \mathcal{O}_{2} \otimes \mathcal{O}_{3} $ of the $ C^{\ast} $-algebras  $ \mathcal{O}_{2}$ and  $\mathcal{O}_{3} $. Suppose that for some $ \beta \in \R $, there is a state $ \omega $ on this tensor product $ C^{\ast} $-algebra that satisfies the KMS condition for the automorphism group $ \alpha $. Then by the Olesen-Pedersen result, $ \omega $ restricted to the $ C^{\ast} $-subalgebra $ \mathcal{O}_{2} \otimes \C \Id_{\mathcal{O}_{3}} $ satisfies the KMS condition for $ \alpha $ only at $ \beta = \func{\ln}{2} $, whereas $ \omega $ restricted to the subalgebra $ \C \Id_{\mathcal{O}_{2}} \otimes \mathcal{O}_{3} $ satisfies the KMS condition for $ \alpha $ only at $ \beta = \func{\ln}{3} $. Therefore, by  the aforementioned result of Olsen and Petersen \cite{OP}, there cannot be any KMS states for the $ C^{\ast} $-algebra $ \mathcal{O}_{2} \otimes \mathcal{O}_{3} $ associated to $ \Trip{X}{\sigma}{\varphi} $ for the automorphism group $ \alpha = \alpha_{1} \otimes \alpha_{2} $.
\end{Eg}

% ---------------------------------------------------------------------------------------------------------------------------------------------

We are now in a position to introduce the generalized gauge dynamics of a $ k $-Ruelle dynamical system, which satisfies the conditions of Definition \ref{RPF Ruelle Dynamical System}. \\

% ---------------------------------------------------------------------------------------------------------------------------------------------

\begin{Def}[Generalized Gauge Dynamics] \label{Generalized Gauge Dynamics of an RPF Ruelle Dynamical System}
The \emph{generalized gauge dynamics} of a $ k $-Ruelle  dynamical system $ \Trip{X}{\sigma}{\varphi} $, which satisfies the conditions of Definition \ref{RPF Ruelle Dynamical System}, is by definition  the $ \R $-dynamical system $ \Pair{\Cstar{\DRG{X}{\sigma}}}{\alpha^{X,\sigma,\varphi}} $ defined by:
$$
\Func{\func{\alpha^{X,\sigma,\varphi}_{t}}{f}}{\gamma} \df e^{i t \func{c_{X,\sigma,\varphi}}{\gamma}} \func{f}{\gamma},
$$
for all $ f \in \Cc{\DRG{X}{\sigma}} $, $ \gamma \in {\DRG{X}{\sigma}} $ and $ t \in \R $, where here we are implicitly using the canonical embedding of $ \Cc{\DRG{X}{\sigma}} $ into $ \CstarRed{\DRG{X}{\sigma}} $.
\end{Def}

The following result may be found in \cite{Re}; see also the discussion preceding Proposition 3.2 of \cite{KR}. \\

% ---------------------------------------------------------------------------------------------------------------------------------------------

\begin{Prop}[\cite{Re}] \label{Constructing KMS States from Quasi-Invariant Measures}
Let $ \Trip{X}{\sigma}{\varphi} $ be a $ k $-Ruelle dynamical system that satisfies the conditions of Definition \ref{RPF Ruelle Dynamical System}, and let $ \beta \in \R $. Then for every quasi-invariant measure $ \mu $ for $ \Pair{X}{\sigma} $ with  continuous Radon-Nikodym derivative $ e^{- \beta c_{X,\sigma,\varphi}} $, there exists a $ \KMS{\beta} $-state $ \omega $ for the generalized gauge dynamics of $ \Trip{X}{\sigma}{\varphi} $ that is uniquely determined by:
$$
\func{\omega}{f} = \Int{X}{\func{f}{x,0,x}}{\mu},
$$
for all $ f \in \Cc{\DRG{X}{\sigma}} $.
\end{Prop}

% ---------------------------------------------------------------------------------------------------------------------------------------------

It is not necessarily the case that every $ \KMS{\beta} $-state for the generalized gauge dynamics of a $ k $-Ruelle  dynamical system, which satisfies the conditions of Definition \ref{RPF Ruelle Dynamical System}, $ \Trip{X}{\sigma}{\varphi} $ originates from a quasi-invariant measure for $ \Pair{X}{\sigma} $ with $ e^{- \beta c_{X,\sigma,\varphi}} $ as a continuous Radon-Nikodym derivative, as  described above. However, A. Kumjian and J. Renault showed in \cite[ Proposition 3.2]{KR} that this is indeed the case if $ \InvIm{c_{X,\sigma,\varphi}}{\SSet{0}} $ is a principal sub-groupoid of $ \DRG{X}{\sigma} $.

Using  Proposition \ref{Constructing KMS States from Quasi-Invariant Measures}, we can now prove the following result. \\

% ---------------------------------------------------------------------------------------------------------------------------------------------

\begin{Thm} \label{thm:suff-cond-existence-KMS-states}
Let $ \Trip{X}{\sigma}{\varphi} $ be a $k$--Ruelle dynamical system, which  has a unique eigenmeasure as in Definition \ref{RPF Ruelle Dynamical System}, and $ \beta \in \R \setminus \SSet{0} $ be such that $ \func{\Br{\Ruelle{X}{\sigma_{i}}{\beta \varphi_{i}}}^{\ast}}{\mu} = \mu $ for each $ i \in \SqBr{k} $, where $ \mu $ is an eigenmeasure for the corresponding Ruelle operator with eigenvalue $ 1 $, see Definition \ref{Ruelle Triple}(2). Then there exists a $ \KMS{\beta} $ state as in Proposition \ref{Constructing KMS States from Quasi-Invariant Measures}   for the generalized gauge dynamics corresponding to $ \Trip{X}{\sigma}{\beta \varphi} $.
\end{Thm}

% ---------------------------------------------------------------------------------------------------------------------------------------------

Even if the $ k $-Ruelle dynamical system does not satisfy the conditions of Definition \ref{RPF Ruelle Dynamical System} so that there does not exist an eigenmeasure for the dual of the Ruelle operator with eigenvalue $ 1 $, we can sometimes modify the $ 1 $-cocycle $ \varphi $ to obtain a new $ 1 $-cocycle $ {\varsigma} $ that does satisfy those hypotheses. The next result was motivated by \cite[Proposition 4.4]{FGLP}, which was in turn based on \cite[Remark 5.25 and Proposition 5.8]{McN}. \\

% -------------------------------------------------------------------------------------------------------------------------

\begin{Cor} \label{Existence of KMS States for Generalized Gauge Dynamics of RPF Ruelle Dynamical Systems}
Let $ \Trip{X}{\sigma}{\varphi} $ be a $ k $-Ruelle dynamical system satisfying the conditions in Definition \ref{RPF Ruelle Dynamical System}. For $ i \in \SqBr{k} $ and $ x \in X $, define
$$
\func{\varsigma_{i}}{x} \df \func{\ln}{\lambda^{X,\sigma,\varphi}_{i}} - \func{\varphi_{i}}{x}
\qquad \text{and} \qquad
\varsigma \df \Seq{\varsigma_{i}}{i \in \SqBr{k}}.
$$
Then $ \Trip{X}{\sigma}{\varsigma} $ is a $ k $-Ruelle dynamical system  and $ \mu^{X,\sigma,\varsigma} $ is a quasi-invariant measure for $ \Pair{X}{\sigma} $, with continuous Radon-Nikodym derivative $ e^{- c_{X,\sigma,\varsigma}} $. Moreover $ \mu^{X,\sigma,\varsigma} = \mu^{X,\sigma,\varphi} $ so that $ \mu^{X,\sigma,\varsigma} $ corresponds by Proposition \ref{Constructing KMS States from Quasi-Invariant Measures} to a KMS state for the generalized gauge dynamics of $ \Trip{X}{\sigma}{\varsigma} $.
\end{Cor}

\begin{proof}
Since the $ \Seq{\varphi_{i}}{i \in \SqBr{k}} $ satisfies the cocycle condition, it is easily checked that 
 $ \Seq{ \varsigma_{i}}{i \in \SqBr{k}} = \Seq{\func{\ln}{\lambda^{X,\sigma,\varphi}_{i}}- \varphi_{i} }{i \in \SqBr{k}} $ satisfies the cocycle condition, so that $ \Trip{X}{\sigma}{\varsigma} $ is a $ k $-Ruelle dynamical system. Similarly, $ \Trip{X}{\sigma}{- \varsigma} $ is a $ k $-Ruelle dynamical system too.

Suppose that $ \Seq{\alpha_{i}}{i \in \SqBr{k}} $ is a $ k $-tuple in $ \N_{> 0} $ and $ \nu $ is a Borel probability measure on $ X $ such that $ \func{\Br{\Ruelle{X}{\sigma_{i}}{-  \varsigma_{i}}}^{\ast}}{\nu} = \alpha_{i} \nu $ for all $ i \in \SqBr{k} $. Then for each $ i \in \SqBr{k} $, the equalities
$$
	\Ruelle{X}{\sigma_{i}}{- \varsigma_{i}}
= \Ruelle{X}{\sigma_{i}}{\varphi_{i} - \func{\ln}{\lambda^{X,\sigma,\varphi}_{i}}}
= e^{- \func{\ln}{\lambda^{X,\sigma,\varphi}_{i}}}  \Ruelle{X}{\sigma_{i}}{\varphi_{i}}
= \frac{1}{\lambda^{X,\sigma,\varphi}_{i}} \Ruelle{X}{\sigma_{i}}{\varphi_{i}},
$$
imply
$$
  \func{\Br{\Ruelle{X}{\sigma_{i}}{\varphi_{i}}}^{\ast}}{\nu}
= \func{\Br{\lambda^{X,\sigma,\varphi}_{i}  \Ruelle{X}{\sigma_{i}}{- \varsigma_{i}}}^{\ast}}{\nu}
= \lambda^{X,\sigma,\varphi}_{i} \func{\Br{\Ruelle{X}{\sigma_{i}}{- \varsigma_{i}}}^{\ast}}{\nu}
= \lambda^{X,\sigma,\varphi}_{i} \alpha_{i} \nu.
$$
As $ \Trip{X}{\sigma}{\varphi} $ satisfies Definition \ref{RPF Ruelle Triple}, it follows that $ \nu = \mu^{X,\sigma,\varphi} $ and $ \alpha_{i} = 1 $ for each $ i \in \SqBr{k} $, so $ \Trip{X}{\sigma}{- \varsigma} $ satisfies Definition  \ref{RPF Ruelle Triple} too. Moreover by Definition  \ref{RPF Ruelle Dynamical System} we have $ \mu^{X,\sigma,\varsigma} = \mu^{X,\sigma,\varphi} $.

Now, as $ \func{\Br{\Ruelle{X}{\sigma_{i}}{- \varsigma_{i}}}^{\ast}}{\mu^{X,\sigma,\varphi}} = \mu^{X,\sigma,\varphi} $ for all $ i \in \SqBr{k} $, Theorem \ref{A Characterization of Solutions of the Radon-Nikodym Problem} tells us that $ \mu^{X,\sigma,\varphi} $ is quasi-invariant for $ \Pair{X}{\sigma} $, with continuous Radon-Nikodym derivative $ e^{c_{X,\sigma,- \varsigma}} = e^{- c_{X,\sigma,\varsigma}} $. Therefore, by Proposition \ref{Constructing KMS States from Quasi-Invariant Measures}, $ \mu^{X,\sigma,\varphi} $ corresponds to a KMS-state for the generalized gauge dynamics of $ \Trip{X}{\sigma}{\varsigma} $.
\end{proof}

%-------------------------------------------------------------------------------------------------------------------------

The following corollary is thus clear. \\

%-------------------------------------------------------------------------------------------------------------------------

\begin{Cor} \label{cor:Existence of KMS States for Generalized Gauge Dynamics of RPF Ruelle Dynamical Systems}
Let $ \Trip{X}{\sigma}{\varphi} $ be a $ k $-Ruelle dynamical system, which satisfies the conditions of Definition \ref{RPF Ruelle Dynamical System}, and let $ \beta \in \R\setminus \{ 0\} $. Define (with notation as in Definition \ref{RPF Ruelle Dynamical System}) for $ i \in \SqBr{k} $ and $ x \in X $
$$
\func{\varsigma_{i}}{x} \df \frac{\func{\ln}{\lambda^{X,\sigma,\varphi}_{i}} - \func{\varphi_{i}}{x}}{\beta}
\qquad \text{and} \qquad
\varsigma \df \Seq{\varsigma_{i}}{i \in \SqBr{k}}.
$$
Then $ \Trip{X}{\sigma}{\varsigma} $ also satisfies the conditions of Definition \ref{RPF Ruelle Dynamical System} and $ \mu^{X,\sigma,\varsigma} = \mu^{X,\sigma,\varphi} $ is a quasi-invariant measure for $ \Pair{X}{\sigma} $, with continuous Radon-Nikodym derivative $ e^{- \beta c_{X,\sigma,\varsigma}} $. Consequently, $ \mu^{X,\sigma,\varphi} $ corresponds by Proposition \ref{Constructing KMS States from Quasi-Invariant Measures} to a $ \KMS{\beta} $-state for the generalized gauge dynamics of $ \Trip{X}{\sigma}{\varsigma} $.
\end{Cor}

%%%%%%%%%%%%%%%%%%%%%%%%%%%%%%%%%%%%%%%%%%%%%%%%%%%%%%%%%%%%%%%%%%%%%%%%%%%%%%%%%%%%%%%%%%%%%%%%%%%%%%%%%%%%%%%%%%%%%%%%%%%%%%%%%%%%%%%%%%%%%%%
\section{KMS states associated to higher-rank graphs}\label{sec:gene-gauge-dyn-higher-rank-graphs} %%%%%%%%%%%%%%%%%%%%%%%%%%%%%%%%%%%%%%%%%%%%
%%%%%%%%%%%%%%%%%%%%%%%%%%%%%%%%%%%%%%%%%%%%%%%%%%%%%%%%%%%%%%%%%%%%%%%%%%%%%%%%%%%%%%%%%%%%%%%%%%%%%%%%%%%%%%%%%%%%%%%%%%%%%%%%%%%%%%%%%%%%%%%

In this section, we shall use the results obtained thus far to answer existence-uniqueness questions on KMS states for generalized gauge dynamics associated to finite higher-rank graphs.

In what follows, $ \N^{k} $ is viewed as a countable category with a single object $ 0 $ and composition of morphisms implemented by $ + $. \\

% ---------------------------------------------------------------------------------------------------------------------------------------------

\begin{Def}[$ k $-Graphs \cite{KP1}] \label{Higher-Rank Graph}
A \emph{higher-rank graph $ \Lambda $ of rank $ k $} or, more briefly, a \emph{$ k $-graph} is a countable category $ \Lambda $ equipped with a functor $ d: \Lambda \to \N^{k} $ --- called the \emph{degree functor} --- such that  the factorization property holds: for every $ \lambda \in \Lambda $ and $ m,n \in \N^{k} $ such that $ \func{d}{\lambda} = m + n $, there are unique $ \mu,\nu \in \Lambda $ that satisfy the following conditions:
\begin{enumerate}
\item[(1)]
$ \func{d}{\mu} = m $ and $ \func{d}{\nu} = n $.

\item[(2)]
$ \lambda = \mu \nu $.
\end{enumerate}

For notational convenience, we will adopt the following $ k $-graph-theoretic terminology. Given a $ k $-graph $ \Lambda $ with degree functor $ d $, for each $ n \in \N^{k} $, let $ \Lambda^{n} \df \InvIm{d}{\SSet{n}} $. The elements of $ \Lambda^{0} $ are called the \emph{vertices} of $ \Lambda $, and it can be shown that $ \Obj{\Lambda} = \Lambda^{0} $. The elements of $ \Lambda^{\mathbf{e}_{i}} $, for $ \mathbf{e}_{i} $ a canonical generator of $\N^{k} $, are called the \emph{edges} of $ \Lambda $. Also, let
\begin{gather*}
v \Lambda       \df \Set{\lambda \in \Lambda}{\func{r}{\lambda} = v}, \qquad
v \Lambda^{n}   \df \Set{\lambda \in \Lambda^{n}}{\func{r}{\lambda} = v}, \\
v \Lambda w     \df \Set{\lambda \in \Lambda}{\func{s}{\lambda} = w ~ \text{and} ~ \func{r}{\lambda} = v}, \qquad
v \Lambda^{n} w \df \Set{\lambda \in \Lambda^{n}}{\func{s}{\lambda} = w ~ \text{and} ~ \func{r}{\lambda} = v}.
\end{gather*}

A $ k $-graph $ \Lambda $ is called \emph{finite} if $ \Card{\Lambda^{n}} < \infty $ for all $ n \in \N^{k} $; $ \Lambda $ is said to be \emph{source-free} if $ v \Lambda^{n} \neq \varnothing $ for all $ n \in \N^{k} $ and $ v \in \Lambda^{0} $; and $ \Lambda $ is said to be \emph{row-finite} if $ v \Lambda^{n}$ is finite for all $ n \in \N^{k} $ and $ v \in \Lambda^{0} $.

Moreover, a \emph{$ k $-graph morphism} from a $ k $-graph $ \Lambda $ to another $ \Lambda' $ is a degree-preserving functor $ f: \Lambda \to \Lambda' $. \\
\end{Def}

% ---------------------------------------------------------------------------------------------------------------------------------------------

\begin{Def}[Strong Connectivity and Primitivity \cite{ALRS,KP1,KP2}] \label{Source-Freeness, Finiteness, Strong Connectivity, and Primitivity}
Let $ \Lambda $ be a $ k $-graph. Then $ \Lambda $ is said to be \emph{strongly connected} if  $ v \Lambda w \neq \varnothing $ for all $ v,w \in \Lambda^{0} $, while $ \Lambda $ is said to be \emph{primitive} if there is an $ n \in \N^{k} \setminus \SSet{0} $ such that $ v \Lambda^{n} w \neq \varnothing $ for all $ v,w \in \Lambda^{0} $. Evidently, primitivity is a stronger condition than strong connectivity. \\
\end{Def}

\begin{Rmk}
Note that $ \N^{k} $ may itself be regarded as a $ k $-graph with one vertex. It is called the \emph{trivial $ k $-graph} and is both finite and primitive. \\
\end{Rmk}

\begin{Eg}[see Example 1.7(ii) of \cite{KP1}] \label{cotrivial k-graph}
Consider the countable category $ \Omega_{k} $ whose underlying set is
$$
\Omega_{k} \df \Set{\Pair{m}{n} \in \N^{k} \times \N^{k}}{m \leq n}
$$
and whose range map, source map, and morphisms are defined as follows:
\begin{itemize}
\item
If $ \Pair{m}{n} \in \Omega_{k} $, then $ \func{s}{m,n} \df \Pair{n}{n} $ and $ \func{r}{m,n} \df \Pair{m}{m} $, so that $ \Pair{\Pair{k}{l}}{\Pair{m}{n}} \in \Omega_{k}^{2} $ is composable if and only if $ l = m $.

\item
If $ \Pair{l}{m},\Pair{m}{n} \in \Omega_{k} $, then $ \Pair{l}{m} \Pair{m}{n} \df \Pair{l}{n} $.
\end{itemize}
If we equip $ \Omega_{k} $ with the degree functor $ d: \Omega_{k} \to \N^{k} $ defined by: $ \func{d}{m,n} \df n - m $ for $ \Pair{m}{n} \in \Omega_{k} $, then $ \Omega_{k} $ is a $ k $-graph.  Note that $ \Omega_{k} $ is both source-free and row-finite but neither finite nor strongly connected.
\end{Eg}

% ---------------------------------------------------------------------------------------------------------------------------------------------

For the remainder of this section, we shall make the following standing assumptions:
\begin{equation} \label{eq:standing-assumption}
\hbox{The $ k $-graph $ \Lambda $ is source-free, finite, primitive, and non-empty.}
\end{equation}

We will now detail more $ k $-graphs structures. \\

% ---------------------------------------------------------------------------------------------------------------------------------------------

\begin{Def}[Infinite Path Space \cite{KP1}] \label{Infinite-Path Space}
Let $ \Lambda $ be a $ k $-graph satisfying the standing assumptions of \eqref{eq:standing-assumption}. The \emph{infinite-path space} of $ \Lambda $, denoted by $ \Lambda^{\infty} $, is defined by
$$
\Lambda^{\infty} \df \Set{f: \Omega_{k} \to \Lambda}{f ~ \text{is a} ~ k \text{-graph morphism}}.
$$
As $ \Lambda $ is source-free and finite, $ \Lambda^{\infty} $ becomes a non-empty compact Hausdorff space when given the topology generated by the base consisting of cylinder sets, i.e., non-empty compact subsets of the form $ \func{\mathcal{Z}}{\lambda} $ for all $ \lambda \in \Lambda $, where
$$ 
\func{\mathcal{Z}}{\lambda} \df \Set{x \in \Lambda^{\infty}}{\func{x}{0,\func{d}{\lambda}} = \lambda}.
$$
\end{Def}

% ---------------------------------------------------------------------------------------------------------------------------------------------

We can then define a commuting $ k $-tuple $ \sigma = \Seq{\sigma_{i}}{i \in \SqBr{k}} $ of local homeomorphisms of $ \Lambda^{\infty} $ by setting, for all $ i \in \SqBr{k} $, $ x \in \Lambda^{\infty} $, and $ \Pair{m}{n} \in \Omega_{k} $:
$$
\FUNC{\func{\sigma_{i}}{x}}{m,n} \df \func{x}{m + \mathbf{e}_{i},n + \mathbf{e}_{i}}.
$$
We call the $ k $-tuple $ \sigma $ the \emph{shift} on $ \Lambda $, and it is easy to see that, for all $ l \in \N^{k} $, $ x \in \Lambda^{\infty} $, and $ \Pair{m}{n} \in \Omega_{k} $:
$$
\FUNC{\func{\sigma^{l}}{x}}{m,n} = \func{x}{m + l,n + l}.
$$
Furthermore, it can be shown that $ \sigma_{i} $ is surjective for each $ i \in \SqBr{k} $. We refer the reader to \cite{KP1} for details.

We now state the following lemma whose standard proof we omit, see for example \cite[Proposition 2.15]{FGJKP}. \\

% ---------------------------------------------------------------------------------------------------------------------------------------------

\begin{Lem} \label{A Compatible Ultrametric on the Infinite-Path Space}
Let $ \Lambda $ be a $ k $-graph satisfying the standing assumptions of  \eqref{eq:standing-assumption}. Define $ \rho_{\Lambda}: \Lambda^{\infty} \times \Lambda^{\infty} \to \R_{\geq 0} $, where for all $ x,y \in \Lambda^{\infty} $ we set	$ \func{\rho_{\Lambda}}{x,y} \df {2^{- N_{x y}}} $, with
$$
N_{x y} \df \func{\min}{\Set{n \in \N}{\func{x}{n p,\Br{n + 1} p} \neq \func{y}{n p,\Br{n + 1} p}}}.
$$
Here, we have arbitrarily chosen $ p \in \N^{k} $ to satisfy $ \Seq{1}{i \in \SqBr{k}} \leq p $. Furthermore, $ \func{\min}{\varnothing} \df \infty $ by convention. Then $ \rho_{\Lambda} $ is a metric on $ \Lambda^{\infty} $ compatible with the cylinder set topology. \\
\end{Lem}

%---------------------------------------------------------------------------------------------------------------------------------------------

\begin{Lem} \label{Positive Expansivity and Exactness of a Suitable Power of the Shift}
Let $ \Lambda $ be a $ k $-graph satisfying the standing assumptions of  \eqref{eq:standing-assumption}. For any $ p \in N^{k} $ satisfying $ \Seq{1}{i \in \SqBr{k}} \leq p $, the local homeomorphism $ \sigma^{p} $ is positively expansive and exact.
\end{Lem}

\begin{proof}
If $ x,y \in \Lambda^{\infty} $ are distinct, then $ \func{x}{n p,\Br{n + 1} p} \neq \func{y}{n p,\Br{n + 1} p} $ for some $ n \in \N $. From this, one easily verifies that $ \func{\rho_{\Lambda}}{\func{\Br{\sigma^{p}}^{n }}{x},\func{\Br{\sigma^{p}}^{n}}{y}} = 1 $. Hence, $ \sigma^{p} $ is positively expansive.

To prove exactness,  for fixed  $ \lambda \in \Lambda $, we will show that $ \Im{\sigma^{n p}}{\func{\mathcal{Z}}{\lambda}} = \Lambda^{\infty} $ for some $ n \in \N $. For, as  $ \Lambda $ is primitive, there exists  $ q \in \N^{k} \setminus \SSet{0} $ such that $ v \Lambda^{q} w \neq \varnothing $ for all $ v,w \in \Lambda^{0} $. Now choose $ n \in \N $ such that $ \func{d}{\lambda} + q \leq n p $, and $ y \in \Lambda^{\infty} $. As $ \Lambda $ is source-free, there exists   $ \mu \in \func{s}{\lambda} \Lambda^{n p - \func{d}{\lambda} - q} $. Next, for any $ \nu \in \func{s}{\mu} \Lambda^{q} \func{y}{0,0} $,  $ \Trip{\lambda}{\mu}{\nu} $ forms a composable triple. Since $ \func{y}{0,0} = \func{s}{\lambda \mu \nu} $, Proposition 2.3 of \cite{KP1} implies that there exists  $ x \in \Lambda^{\infty} $ such that $ y = \func{\sigma^{\func{d}{\lambda \mu \nu}}}{x} = \func{\sigma^{n p}}{x} $ with
$$
\lambda \mu \nu = \func{x}{0,n p} = \func{x}{0,\func{d}{\lambda}} ~ \func{x}{\func{d}{\lambda},n p - q} ~ \func{x}{n p - q,n p}.
$$
By the factorization property, we get $ \func{x}{0,\func{d}{\lambda}} = \lambda $, so $ x \in \func{\mathcal{Z}}{\lambda} $. Hence, $ y \in \Im{\sigma^{n p}}{\func{\mathcal{Z}}{\lambda}} $, and since $ y \in \Lambda^{\infty} $ is arbitrary, we obtain $ \Im{\sigma^{n p}}{\func{\mathcal{Z}}{\lambda}} = \Lambda^{\infty} $. Therefore, $ \sigma^{n p} $ is exact.
\end{proof}

% ---------------------------------------------------------------------------------------------------------------------------------------------

Let $ \Seq{\varphi_{i}}{i \in \SqBr{k}} $ be a $ k $-tuple of continuous real-valued functions on $ \Lambda^{\infty} $ satisfying the conditions in Definition \ref{The Cocycle Condition} so that if the conditions of Proposition \ref{cor:The Classification Theorem for Continuous Real-Valued 1-Cocycles on Deaconu-Renault Groupoids} and Theorem \ref{The RPF Theorem for Ruelle Dynamical Systems} are satisfied, then $ \Trip{\Lambda^{\infty}}{\sigma}{\varphi} $ will satisfy the conditions in Definition \ref{RPF Ruelle Dynamical System}.

As for Ruelle dynamical systems, there is a version of the RPF Theorem for $ k $-graphs. \\

% ---------------------------------------------------------------------------------------------------------------------------------------------

\begin{Thm} \label{The RPF Theorem for k-Graph Dynamical Systems}
Let $ \Lambda $ be a $ k $-graph satisfying the standing assumptions of  \eqref{eq:standing-assumption}. Assume that $ \varphi = \Seq{\varphi_{i}}{i \in \SqBr{k}} $ is a $ k $-tuple of continuous real-valued  functions on $ \Lambda^{\infty} $ satisfying the cocycle condition, and let $ c_{\varphi} $ denote the associated $ 1 $-cocycle. If there exists a $ p \in \N^{k} $ with $ \Seq{1}{i \in \SqBr{k}} \leq p $ such that $ \func{c_{\varphi}}{p}: \Lambda^\infty \to \R $ is H\"{o}lder-continuous with respect to $ \rho_{\Lambda} $, then the triple $ \Trip{\Lambda^{\infty}}{\sigma}{\varphi} $ satisfies the conditions in Definition \ref{RPF Ruelle Dynamical System}.
\end{Thm}

\begin{proof}
By Theorem \ref{The Walters Criteria and the Jiang-Ye Criterion Imply the RPF Property}, Lemma \ref{A Compatible Ultrametric on the Infinite-Path Space}, and Lemma \ref{Positive Expansivity and Exactness of a Suitable Power of the Shift}, the Ruelle triple $ \Trip{\Lambda^{\infty}}{\sigma^{p}}{\func{c_{\varphi}}{p}} $ satisfies the conditions in Definition  \ref{The Walters Criteria and the Jiang-Ye Conditions}, so by Theorem \ref{The RPF Theorem for Ruelle Dynamical Systems} the result follows.
\end{proof}

% ---------------------------------------------------------------------------------------------------------------------------------------------

Note that Proposition \ref{The RPF Theorem for Ruelle Dynamical Systems}, Theorem \ref{thm:RPF-Theorem}, and the hypothesis Theorem \ref{The RPF Theorem for k-Graph Dynamical Systems} will guarantee the existence of a Borel measure $ \mu^{\varphi} $ on $ \Lambda^{\infty} $. We will now establish some useful properties of this measure. \\

% ---------------------------------------------------------------------------------------------------------------------------------------------

\begin{Prop} \label{prop-RPF-k-graphs}
Let $ \Lambda $ be a $ k $-graph satisfying the standing assumptions of  \eqref{eq:standing-assumption}. Suppose that $ \Trip{\Lambda^{\infty}}{\sigma}{\varphi} $, with $ \varphi = \Seq{\varphi_{i}}{i \in \SqBr{k}} $ a $ k $-tuple of  continuous real-valued  functions on $ \Lambda^{\infty} $   satisfying the cocycle condition. Let $ c_{\varphi} $ denote the associated $ 1 $-cocycle. If $ \Trip{\Lambda^{\infty}}{\sigma}{\varphi} $ is a $ k $-Ruelle dynamical system satisfying the conditions in Definition \ref{RPF Ruelle Dynamical System}, then for all $ \lambda \in \Lambda $,
$$
  \func{\mu^{\varphi}}{\ZZintfor{\lambda}}
= \Br{\bm{\lambda}^{\varphi}}^{- \func{d}{\lambda}} \Int{\ZZintfor{\func{s}{\lambda}}}{e^{\func{c_{\varphi}}{\func{d}{\lambda}}}}{\mu^{\varphi}}{x},
$$
where to simplify the notation we denoted by  $ \bm{\lambda}^{\varphi} $ (resp. $\mu^\varphi$) 
the $ k $-tuple of eigenvalues  (resp. the eigenmeasure) associated  to the k-Ruelle dynamical system $ \Trip{\Lambda^{\infty}}{\sigma}{\varphi} $.
\end{Prop}

\begin{proof}
Fix an arbitrary $ \lambda \in \Lambda $. For every $ i \in \SqBr{k} $, we have $ \func{\Br{\Ruelle{\Lambda^{\infty}}{\sigma_{i}}{\varphi_{i}}}^{\ast}}{\mu^{\varphi}} = \lambda^{\varphi}_{i}  \mu^{\varphi} $. Hence,
$$
  \func{\Br{\Ruelle{\Lambda^{\infty}}{\sigma^{\func{d}{\lambda}}}{\func{c_{\varphi}}{\func{d}{\lambda}}}}^{\ast}}{\mu^{\varphi}}
= \Br{\bm{\lambda}^{\varphi}}^{\func{d}{\lambda}} \mu^{\varphi}.
$$
so integrating $ 1_{\ZZintfor{\lambda}} \in \Cont{\Lambda^{\infty},\R} $ with respect to the equal measures on the left and right hand side of the above equation and using the definition of 
$ \Br{\Ruelle{\Lambda^{\infty}}{\sigma^{\func{d}{\lambda}}}{\func{c_{\varphi}}{\func{d}{\lambda}}}}^{\ast} $ yields
$$
	\IInt{\Lambda^{\infty}}
	     {\SqBr{\sum_{\substack{y \in \Lambda^{\infty} \\ \func{\sigma^{\func{d}{\lambda}}}{y} = x}} e^{\FUNC{\func{c_{\varphi}}{\func{d}{\lambda}}}{y}} \func{1_{\ZZintfor{\lambda}}}{y}}}
		   {\mu^{\varphi}}{x}
= \Br{\bm{\lambda}^{\varphi}}^{\func{d}{\lambda}} \func{\mu^{\varphi}}{\ZZintfor{\lambda}}.
$$
If $ x \in \Lambda^{\infty} \setminus \ZZintfor{\func{s}{\lambda}} $, then there does not exist a $ y \in \ZZintfor{\lambda} $ such that $ \func{\sigma^{\func{d}{\lambda}}}{y} = x $, so
$$
\sum_{\substack{y \in \Lambda^{\infty} \\ \func{\sigma^{\func{d}{\lambda}}}{y} = x}} e^{\FUNC{\func{c_{\varphi}}{\func{d}{\lambda}}}{y}} \func{1_{\ZZintfor{\lambda}}}{y} = 0.
$$
It  follows that:
$$
	\IInt{\Lambda^{\infty}}
		   {\SqBr{\sum_{\substack{y \in \Lambda^{\infty} \\ \func{\sigma^{\func{d}{\lambda}}}{y} = x}} e^{\FUNC{\func{c_{\varphi}}{\func{d}{\lambda}}}{y}} \func{1_{\ZZintfor{\lambda}}}{y}}}
		   {\mu^{\varphi}}{x}
= \IInt{\ZZintfor{\func{s}{\lambda}}}
		   {\SqBr{\sum_{\substack{y \in \Lambda^{\infty} \\ \func{\sigma^{\func{d}{\lambda}}}{y} = x}} e^{\FUNC{\func{c_{\varphi}}{\func{d}{\lambda}}}{y}} \func{1_{\ZZintfor{\lambda}}}{y}}}
		   {\mu^{\varphi}}{x}.
$$
Given an $ x \in \ZZintfor{\func{s}{\lambda}} $, there exists precisely one $ y \in \ZZintfor{\lambda} $ such that $ \func{\sigma^{\func{d}{\lambda}}}{y} = x $, namely, $ \lambda x $. Consequently,
$$
	\IInt{\ZZintfor{\func{s}{\lambda}}}
	     {\SqBr{\sum_{\substack{y \in \Lambda^{\infty} \\ \func{\sigma^{\func{d}{\lambda}}}{y} = x}}e^{\FUNC{\func{c_{\varphi}}{\func{d}{\lambda}}}{y}} \func{1_{\ZZintfor{\lambda}}}{y}}}
			 {\mu^{\varphi}}{x}
= \IInt{\ZZintfor{\func{s}{\lambda}}}{e^{\FUNC{\func{c_{\varphi}}{\func{d}{\lambda}}}{\lambda x}} \func{1_{\ZZintfor{\lambda}}}{y}}{\mu^{\varphi}}{x}.
$$
Therefore,
$$
	\IInt{\ZZintfor{\func{s}{\lambda}}}{e^{\FUNC{\func{c_{\varphi}}{\func{d}{\lambda}}}{\lambda x}}}{\mu^{\varphi}}{x}
= \Br{\bm{\lambda}^{\varphi}}^{\func{d}{\lambda}} \func{\mu^{\varphi}}{\ZZintfor{\lambda}},
$$
and a simple rearrangement of terms yields the proposition.
\end{proof}

% ----------------------------------------------------------------------------------------------------------------------------------------------

We list an important positivity property of the measure $ \mu^{\varphi} $ in the event that $ \Trip{\Lambda^{\infty}}{\sigma}{\varphi} $ satisfies the conditions in Definition \ref{RPF Ruelle Dynamical System}. \\

% ----------------------------------------------------------------------------------------------------------------------------------------------

\begin{Cor}
Let $ \Lambda $ be a $ k $-graph satisfying the standing assumptions of  \eqref{eq:standing-assumption}. Suppose that $ \Trip{\Lambda^{\infty}}{\sigma}{\varphi} $ is a $ k $-Ruelle dynamical system, which satisfies the conditions  of Definition \ref{RPF Ruelle Dynamical System}. Then $ \func{\mu^{\varphi}}{\ZZintfor{\lambda}} > 0 $ for every $ \lambda \in \Lambda $.
\end{Cor}

\begin{proof}
Suppose by way of contradiction that there exists $ \lambda \in \Lambda $ such that $ \func{\mu^{\varphi}}{\ZZintfor{\lambda}} = 0 $. By Lemma \ref{Positive Expansivity and Exactness of a Suitable Power of the Shift}, if $ \Seq{1}{i \in \SqBr{k}} \leq p $, then $ \sigma^{p} $ is positively expansive and exact. But then there exists $ n \in \N $ such that $ \Im{\sigma^{n p}}{\ZZintfor{\lambda}} = \Lambda^{\infty} $. We now use Proposition 4.2 of \cite{Re1} again to deduce that
$$
  \func{\mu^{\varphi}}{\Lambda^{\infty}}
= \func{\mu^{\varphi}}{\Im{\sigma^{n p}}{\ZZintfor{\lambda}}}
=	\func{\mu^{\varphi}}{\ZZintfor{\lambda}}
= 0.
$$
Since  $ \func{\mu^{\varphi}}{\Lambda^{\infty}} = 1 $, this gives us a contradiction.
\end{proof}

% ----------------------------------------------------------------------------------------------------------------------------------------------

Recall that if $ \Lambda $ is a $ k $-graph, and $ H $ is an abelian group, a map $ h: \Lambda \to H $ is called a \emph{categorical $ 1 $-cocycle} if $ \func{h}{\lambda \mu} = \func{h}{\lambda} + \func{h}{\mu} $ whenever $ \Pair{\lambda}{\mu} $ is composable. In the case when $ H = \R $ and the image of $ h $ lies entirely inside the nonnegative real numbers, $ h $ was called an ``$ \R^{+} $-functor'' in \cite{FGLP}.

Next, observe that, if $ h $ is a categorical $ 1 $-cocycle taking values in $ \R $, then a routine calculation shows that the $ k $-tuple of functions $ \varphi^{h,\theta} \df \Seq{\varphi^{h,\theta}_{i}}{i \in \SqBr{k}} $, where $ \varphi^{h,\theta}_{i}: \Lambda^{\infty} \to \R $ is defined for all $ i \in \SqBr{k} $ and $ x \in \Lambda^{\infty} $ by
\begin{equation}\label{def:def of theta h cocycle}
\func{\varphi^{h,\theta}_{i}}{x} = - \theta \func{h}{\func{x}{0,\mathbf{e}_{i}}},
\end{equation}
satisfy the cocycle condition and therefore determine a groupoid $ 1 $-cocycle on $ \DRG{\Lambda^{\infty}}{\sigma} $ taking values in $ \R $ by Lemma \ref{The Workhorse Lemma}; we will call this cocycle $ c_{\varphi^{h,\theta}} $. Hence, $ \Trip{\Lambda^{\infty}}{\sigma}{c_{\varphi^{h,\theta}}} $ is a $ k $-Ruelle dynamical system.

Example \ref{ex:tensor-product-Cuntz-algebra} has shown that an automorphism group on a $ C^{\ast} $-algebra coming from $ \Trip{X}{\sigma}{\varphi} $ need not have a KMS state. However, by using Corollary \ref{Existence of KMS States for Generalized Gauge Dynamics of RPF Ruelle Dynamical Systems} and other results, we can construct a new cocycle from $ \varphi $ giving rise to a different dynamics for which a KMS state does exist. The following theorem was first proved in a different way in Proposition 4.4 in \cite{FGLP}. \\

% ---------------------------------------------------------------------------------------------------------------------------------------------

\begin{Thm}[\cite{FGLP}] \label{KMS States for Generalized Gauge Dynamics of k-Graph Dynamical Systems Corresponding to R+-Functors}
Let $ \Lambda $ be a $ k $-graph satisfying the standing assumptions of  \eqref{eq:standing-assumption}. Let $ h: \Lambda \to \R $ be a nonnegative categorical $ 1 $-cocycle, and let $ \theta  $ be a positive real number. Let $ \varphi^{h,\theta} = \Seq{\varphi^{h,\theta}_i}{i \in \SqBr{k}} $ be  as defined in Equation \eqref{def:def of theta h cocycle}. Then $ \Trip{\Lambda^{\infty}}{\sigma}{\varphi^{h,\theta}} $ is a $ k $-Ruelle dynamical system that satisfies the conditions in Definition \ref{RPF Ruelle Dynamical System}, and for each $ \beta \in \R\setminus \{0\} $, so does
$$
\Trip{\Lambda^{\infty}}{\sigma}{\Seq{\frac{1}{\beta} \Br{\func{\ln}{\lambda^{\varphi^{h,\theta}}_{i}} - \varphi^{h,\theta}_{i}}}{i \in \SqBr{k}}}.
$$
The associated generalized state $ \omega $ on $ \CstarRed{\DRG{\Lambda^{\infty}}{\sigma}} \cong \Cstar{\Lambda} $ uniquely determined by
$$
\forall f \in \Cc{\DRG{\Lambda^{\infty}}{\sigma}}: \qquad
\func{\omega}{f} = \Int{\Lambda^{\infty}}{\func{f}{x,0,x}}{\mu^{\varphi^{h,\theta}}}
$$
is a $ \KMS{\beta} $-state for the dynamics determined by the cocycle $ \varsigma $ given by
$$
\varsigma \df \Seq{\frac{1}{\beta} \Br{\func{\ln}{\lambda^{\varphi^{h,\theta}}_{i}} - \varphi^{h,\theta}_{i}}}{i \in \SqBr{k}};
$$
moreover, $ \mu^{\varsigma} = \mu^{\varphi^{h,\theta}} $.
\end{Thm}

\begin{proof}
Fix an arbitrary $ n \in \N^{k} $. We will first prove that the function $ f: \Lambda^{\infty} \to \R $ defined by, for all $x \in \Lambda^{\infty}$: 
$$
\func{f}{x} \df \func{h}{\func{x}{0,n}}
$$
is H\"{o}lder-continuous with respect to $ \rho_{\Lambda} $. As $ \Lambda^{n} $ is finite, $ h $ clearly achieves both a minimum value $ m $ and a maximum value $ M $ on $ \Lambda^{n} $. Choose $ N \in \N $ such that $ n \leq N p $, with $ \Seq{1}{i \in \SqBr{k}} \leq p $. For $ x,y \in \Lambda^{\infty} $ such that $ \func{\rho_{\Lambda}}{x,y} < \frac{1}{2^{N}} $, then for all $ j \in \SqBr{N} $,
$$
\func{x}{j p,\Br{j + 1} p} = \func{y}{j p,\Br{j + 1} p},
$$
so $ \func{x}{0,N p} = \func{y}{0,N p} $, which yields $ \func{x}{0,n} = \func{y}{0,n} $ by the factorization property. Consequently, for all $x,y \in \Lambda^{\infty}$:
$$
     \Abs{\func{f}{x} - \func{f}{y}}
=    \Abs{\func{h}{\func{x}{0,n}} - \func{h}{\func{y}{0,n}}}
\leq N \Br{M - m} \func{\rho_{\Lambda}}{x,y}.
$$
As $ n \in \N^{k} $ is arbitrary, it follows that $ \varphi^{h,\theta}_{i} $ is H\"{o}lder-continuous with respect to $ \rho_{\Lambda} $ for each $ i \in \SqBr{k} $.

A straightforward calculation demonstrates that $ \func{c_{\varphi^{h,\theta}}}{p} $ is H\"{o}lder-continuous with respect to $ \rho_{\Lambda} $, therefore by applying Theorem \ref{The RPF Theorem for k-Graph Dynamical Systems}, we conclude that $ \Trip{\Lambda^{\infty}}{\sigma}{\varphi^{h,\theta}} $ satisfies the conditions in Definition \ref{RPF Ruelle Dynamical System}.
			
Applying Theorem \ref{The RPF Theorem for k-Graph Dynamical Systems} we conclude that $ \Trip{\Lambda^{\infty}}{\sigma}{\varphi^{h,\theta}} $ satisfies the conditions in Definition \ref{RPF Ruelle Dynamical System}.

It now follows from Corollary \ref{Existence of KMS States for Generalized Gauge Dynamics of RPF Ruelle Dynamical Systems} that
$$
\Trip{\Lambda^{\infty}}{\sigma}{\Seq{\frac{1}{\beta} \Br{\func{\ln}{\lambda^{\varphi^{h,\theta}}_{i}} - \varphi^{h,\theta}_{i}}}{i \in \SqBr{k}}}
$$
also satisfies the conditions in Definition \ref{RPF Ruelle Dynamical System}, with $ \mu^{\varsigma} = \mu^{\varphi^{h,\theta}} $ by Corollary \ref{Existence of KMS States for Generalized Gauge Dynamics of RPF Ruelle Dynamical Systems}. By Proposition \ref{Constructing KMS States from Quasi-Invariant Measures}, the state $ \omega $ on $ \Cstar{\DRG{\Lambda^{\infty}}{\sigma}} \cong \Cstar{\Lambda} $ that is uniquely determined by, for all $f \in \Cc{\DRG{\Lambda^{\infty}}{\sigma}}$:
$$
\func{\omega}{f} = \Int{\Lambda^{\infty}}{\func{f}{x,0,x}}{\mu^{\varphi^{h,\theta}}},
$$
with notation as in Proposition \ref{prop-RPF-k-graphs}, 
is a $ \KMS{\beta} $-state for the generalized gauge dynamics of this particular dynamical system.
\end{proof}

% ---------------------------------------------------------------------------------------------------------------------------------------------

The following corollary gives more information about the eigenmeasure $ \mu^{\varphi^{h,\theta}} $ and relates it to the eigenvalues of the Ruelle-Perron-Frobenius operator. \\

% ---------------------------------------------------------------------------------------------------------------------------------------------

\begin{Cor} \label{cor:k-graph-cor}
Let $ \Lambda $ be a $ k $-graph satisfying the standing assumptions of \eqref{eq:standing-assumption}, and let $ h: \Lambda \to \R_{\geq 0} $, $ \theta \in \R \setminus \{0\} $, and $ \varphi^{h,\theta} $ be as in Theorem \ref{KMS States for Generalized Gauge Dynamics of k-Graph Dynamical Systems Corresponding to R+-Functors}; let $ c_{\varphi^{h,\theta}} $ denote the associated cocycle. Then for all $ \lambda \in \Lambda $,
$$
  \func{\mu^{\varphi^{h,\theta}}}{\ZZintfor{\lambda}}
= \Br{\bm{\lambda}^{\varphi^{h,\theta}}}^{- \func{d}{\lambda}} e^{- \theta \func{h}{\lambda}} \func{\mu^{\varphi^{h,\theta}}}{\ZZintfor{\func{s}{\lambda}}},
$$
where $ \varphi^{h,\theta} $ is the $ k $-tuple of elements of  $ \Cont{\Lambda^{\infty},\R} $ that is defined for all $ i \in \SqBr{k} $ and $ x \in \Lambda^{\infty} $ by
$$
\func{\varphi^{h,\theta}_{i}}{x} \df - \theta \func{h}{\func{x}{0,\mathbf{e}_{i}}}.
$$
\end{Cor}

\begin{proof}
We already know from Theorem 7.9 that for all $ n \in \N^{k} $ and $ x \in \Lambda^{\infty} $,
$$
\func{c_{\varphi^{h,\theta}}}{n} = - \theta \func{h}{\func{x}{0,n}}.
$$
Hence, for every $ \lambda \in \Lambda $, if $ x \in \ZZintfor{\func{s}{\lambda}} $, we have:
$$
	\FUNC{\func{c_{\varphi^{h,\theta}}}{\func{d}{\lambda}}}{\lambda x}
= - \theta \func{h}{\Func{\lambda x}{0,\func{d}{\lambda}}}
= - \theta \func{h}{\lambda}.
$$
Consequently, by Proposition \ref{prop-RPF-k-graphs}, for all $ \lambda \in \Lambda $,
\begin{align*}
		\func{\mu^{\varphi^{h,\theta}}}{\ZZintfor{\lambda}}
& = \Br{\bm{\lambda}^{\varphi^{h,\theta}}}^{- \func{d}{\lambda}}
    \IInt{\ZZintfor{\func{s}{\lambda}}}{e^{\FUNC{\func{c_{\varphi^{h,\theta}}}{\func{d}{\lambda}}}{\lambda x}}}{\mu^{\varphi^{h,\theta}}}{x} \\
& = \Br{\bm{\lambda}^{\varphi^{h,\theta}}}^{- \func{d}{\lambda}} \IInt{\ZZintfor{\func{s}{\lambda}}}{e^{- \theta \func{h}{\lambda}}}{\mu^{\varphi^{h,\theta}}}{x} \\
& = \Br{\bm{\lambda}^{\varphi^{h,\theta}}}^{- \func{d}{\lambda}} e^{- \theta \func{h}{\lambda}} \IInt{\ZZintfor{\func{s}{\lambda}}}{1}{\mu^{\varphi^{h,\theta}}}{x} \\
& = \Br{\bm{\lambda}^{\varphi^{h,\theta}}}^{- \func{d}{\lambda}} e^{- \theta \func{h}{\lambda}} \func{\mu^{\varphi^{h,\theta}}}{\ZZintfor{\func{s}{\lambda}}}.
\end{align*}
The corollary is therefore proven.
\end{proof}

%%%%%%%%%%%%%%%%%%%%%%%%%%%%%%%%%%%%%%%%%%%%%%%%%%%%%%%%%%%%%%%%%%%%%%%%%%%%%%%%%%%%%%%%%%%%%%%%%%%%%%%%%%%%%%%%%%%%%%%%%%%%%%%%%%%%%%%%%%%%%%%
% REFERENCES %%%%%%%%%%%%%%%%%%%%%%%%%%%%%%%%%%%%%%%%%%%%%%%%%%%%%%%%%%%%%%%%%%%%%%%%%%%%%%%%%%%%%%%%%%%%%%%%%%%%%%%%%%%%%%%%%%%%%%%%%%%%%%%%%%
%%%%%%%%%%%%%%%%%%%%%%%%%%%%%%%%%%%%%%%%%%%%%%%%%%%%%%%%%%%%%%%%%%%%%%%%%%%%%%%%%%%%%%%%%%%%%%%%%%%%%%%%%%%%%%%%%%%%%%%%%%%%%%%%%%%%%%%%%%%%%%%

%%%%%%%%%%%%%%%%%%%%%%%%%%%%%%%%%%%%%%%%%%%%%%%%%%%%%%%%%%%%%%%%%%%%%%%%%%%%%%%%%%%%%%%%%%%%%%%%%%%%%%%%%%%%%%%%%%%%%%%%%%%%%%%%%%%%%%%%%%%%%%%
% END DOCUMENT %%%%%%%%%%%%%%%%%%%%%%%%%%%%%%%%%%%%%%%%%%%%%%%%%%%%%%%%%%%%%%%%%%%%%%%%%%%%%%%%%%%%%%%%%%%%%%%%%%%%%%%%%%%%%%%%%%%%%%%%%%%%%%%%
%%%%%%%%%%%%%%%%%%%%%%%%%%%%%%%%%%%%%%%%%%%%%%%%%%%%%%%%%%%%%%%%%%%%%%%%%%%%%%%%%%%%%%%%%%%%%%%%%%%%%%%%%%%%%%%%%%%%%%%%%%%%%%%%%%%%%%%%%%%%%%%

\end{document}